\documentclass[10pt]{amsart}
\usepackage{amssymb}

\newtheorem{proposition}{Proposition}[section]
\newtheorem{lemma}[proposition]{Lemma}
\newtheorem{convention}[proposition]{Convention}
\newtheorem{corollary}[proposition]{Corollary}
\newtheorem{theorem}[proposition]{Theorem}
\newtheorem{remark}[proposition]{Remark}

\theoremstyle{definition}

\newcommand{\selabel}[1]{\label{se:#1}}
\newcommand{\seref}[1]{Section~\ref{se:#1}}

\def\<{\leqslant}
\def\>{\geqslant}
\def\a{\alpha}
\def\b{\beta}
\def\d{\delta}
\def\g{\gamma}

\def\om{\omega}
\def\ol{\overline}

\def\e{\varepsilon}
\def\oo{\infty}

\def\l{\lambda}

\def\s{\sigma}

\def\ti{\times}
\def\ot{\otimes}
\def\ra{\rightarrow}

\date{}

\begin{document}
\title{Representations of Hopf-Ore extensions of group algebras}
\author{Hua Sun}
\address{School of Mathematical Science, Yangzhou University,
	Yangzhou 225002, China}
\email{huasun@yzu.edu.cn}
\author{Hui-Xiang Chen}
\address{School of Mathematical Science, Yangzhou University,
	Yangzhou 225002, China}
\email{hxchen@yzu.edu.cn}
\author{Yinhuo Zhang}
\address{Department of Mathematics $\&$  Statistics, University of Hasselt, Universitaire Campus, 3590 Diepeenbeek, Belgium}
\email{yinhuo.zhang@uhasselt.be}
\begin{abstract}
In this paper, we study the representations of the Hopf-Ore extensions $kG(\chi^{-1}, a, 0)$ of group algebra $kG$, where $k$ is an algebraically closed field. We classify all finite dimensional simple $kG(\chi^{-1}, a, 0)$-modules under the assumption $|\chi|=\infty$ and $|\chi|=|\chi(a)|<\infty$ respectively, and all finite dimensional indecomposable $kG(\chi^{-1}, a, 0)$-modules under the assumption that $kG$ is finite dimensional and semisimple, and $|\chi|=|\chi(a)|$. Moreover, we investigate the decomposition rules for the tensor product modules over $kG(\chi^{-1}, a, 0)$ when char$(k)$=0. Finally, we consider the representations of some Hopf-Ore extension of the dihedral group algebra $kD_n$, where $n=2m$, $m>1$ odd, and char$(k)$=0. The Grothendieck ring and the Green ring of the Hopf-Ore extension are described  respectively in terms of  generators and relations.
\end{abstract}
\maketitle

\section{\bf Introduction}\selabel{1}

During the past years, the classification of Hopf algebras has made great progress.  Andruskiewitsch-Schneider and Angiono-Garcia Iglesias \cite{AndSch2010, AngGar} classified the  finite dimensional pointed Hopf algebras over an algebraically closed field of characteristic zero such that their coradicals are commutative. Beattie et al \cite{MS1, MS2} constructed many pointed Hopf algebras by means of Ore extensions, and answered the tenth Kaplansky's conjecture in the negative. Panov \cite{A.N} introduced Hopf-Ore extensions, and classified the Hopf-Ore extensions of group algebras and the enveloping algebras of Lie algebras. Krop and Radford \cite{LKD} defined the rank of a Hopf algebra to measure the complexity of the Hopf algebras $H$ generated by $H_1$,  and showed that a finite dimensional  rank one pointed Hopf algebra over an algebraically closed field $k$ with char$(k)$=0 is isomorphic to a quotient of a Hopf-Ore extension of its coradical. Scherotzke \cite{SS} proved such a result for the case of char$(k)=p>0$. Wang et al \cite{WangYouChen} generalized the result to the case that $k$ is an arbitrary field. Brown et al \cite{BrOHZhangZhang} studied the connected Hopf algebras and iterated Ore extensions. You et al \cite{YouWangChen} studied generalized Hopf-Ore extension, and classified the generalized Hopf-Ore extensions of the enveloping algebras of some Lie algebras. Zhou et al \cite{ZhouShenLu} proved that every connected graded Hopf algebra with finite GK-dimension over a field $k$ of characteristic zero is an iterated Ore extensions of $k$.

In \cite{A.N}, Panov proved that every Hopf-Ore extension $kG[x; \tau, \d]$ of a group algebra $kG$ is of the form $kG(\chi,a,\delta)$, where $a$ is a central element of the group $G$ and $\chi$ is a linear character of $G$ over the ground field $k$.
If $\chi(a)\neq 1$ then one can assume $\d=0$ by replacing the variable $x$ with $x-\g(1-a)$ for some scalar $\g\in k$, i.e. $kG(\chi, a, \d)\cong kG(\chi, a,0)$, see  \cite{WangYouChen}.  Wang et al  \cite{WangYouChen} also studied the representations of $kG(\chi^{-1}, a, 0)$ and its rank one quotient Hopf algebra $kG(\chi^{-1}, a, 0)/I$. They constructed  finite dimensional indecomposable weight modules over $kG(\chi^{-1}, a, 0)$ and $kG(\chi^{-1}, a, 0)/I$ and classified them. It was shown that there is a simple weight $kG(\chi^{-1}, a, 0)/I$-module $M$ with dim$(M)>1$ only if $|\chi|=|\chi(a)|<\infty$. It is well-known that the finite dimensional representation category mod$H$ of a Hopf algebra $H$ is a tensor category. In \cite{H.H, H.C}, we investigated the decomposition rules for the tensor products of finite dimensional indecomposable weight $kG(\chi^{-1}, a, 0)$-modules and described the structure of the Green ring of the category of finite dimensional weight modules over $kG(\chi^{-1}, a, 0)$ for the case that $k$ is an algebraically closed field of characteristic zero.  This gives rise to the natural questions: How to classify the finite dimensional indecomposable modules over $kG(\chi^{-1}, a, 0)$? How to describe the Green ring of $kG(\chi^{-1}, a, 0)$?

In this paper, we study the finite dimensional representations of $H=kG(\chi^{-1}, a, 0)$,
a Hopf-Ore extension of a group algebra $kG$, where $k$ is an algebraically closed field. The paper is organized as follows. In Section 2, we recall some notions and notations including Grothendieck ring and Green ring, and the Hopf algebra structure of $H$.  Section 3  deals with the finite dimensional irreducible representations of $H$.  We describe and classify the finite dimensional simple modules over $H$ in  two cases:  $|\chi|=\infty$ and $|\chi|=|\chi(a)|<\infty$. In Section 4, we construct and classify the finite dimensional indecomposable $H$-modules under the assumptions that the group algebra $kG$ is  semisimple and $|\chi|=|\chi(a)|$. In Section 5,  we investigate the decomposition rules for tensor product modules over $H$  under the assumptions: ${\rm char}(k)=0$, $|G|<\infty$ and  $|\chi(a)|=|\chi|$. In Section 6,  we apply the obtained results to some Hopf-Ore extension of the group algebra $kD_n$, where $D_n$ is the dihedral group of order $2n$, $n=2m$ with $m>1$  odd, and char$(k)$=0. The Grothedieck ring and the Green ring of the Hopf-Ore extension are described by means of  generators and relations respectively.

\section{\bf Preliminaries}\selabel{2}

Throughout, let $k$ be an algebraically closed field. Unless
otherwise stated, all algebras and Hopf algebras are
defined over $k$; all modules are finite dimensional and left modules;
dim and $\otimes$ denote ${\rm dim}_k$ and $\otimes_k$,
respectively. We refer to \cite{ARS, Ka, Mon} for the basic concepts and notations  of Hopf algebras or those  in the representation theory.  We use $\varepsilon$, $\Delta$ and $S$ to denote the counit,
comultiplication and antipode of a Hopf algebra respectively. Let $k^{\ti}=k\setminus\{0\}$.
For a group $G$, let $\hat{G}$ denote the group of the linear characters of $G$ over $k$, and let $Z(G)$ denote the center of $G$. Let $\mathbb Z$ denote all integers.  $\mathbb N$  stands for  all nonnegative integers, and $\mathbb{N}^+$  stands for all positive integers.  Denote by $\sharp X$  the number of the elements in a set $X$.

\subsection{Grothendieck ring and Green ring}\selabel{2.1}
~~

For an algebra $A$, we denote by mod$A$ the category of finite dimensional $A$-modules.
For a module $M\in{\rm mod}A$ and an element  $n\in\mathbb N$, let $nM$  be the
direct sum of $n$ copies of $M$. Thus  $nM=0$ if $n=0$.

The {\it Grothendieck ring} $G_0(H)$ of a Hopf algebra $H$ is defined to be the abelian group generated by the
isomorphism classes $[V]$ of $V$ in mod$H$ modulo the relations $[V]=[U]+[W]$ for all short exact sequences
$0\ra U\ra V\ra W\ra 0$ in ${\rm mod}H$. The multiplication of $G_0(H)$ is defined by $[U][V]=[U\ot V]$, the tensor product of $H$-modules.
The ring  $G_0(H)$ is  associative  and has  identity. $G_0(H)$ has a $\mathbb Z$-basis $\{[V_i]|i\in I\}$, where $\{V_i|i\in I\}$
are all non-isomorphic simple modules. Moreover, for each $V\in $ mod$A$, we have
$[V]=\sum_i[V:V_i][V_i]$ in $G_0(H)$, where $[V:V_i]$ denotes the multiplicity of $V_i$ in a composition series of $V$.

The {\it Green ring} $r(H)$ of a Hopf algebra $H$ is defined to be the abelian group generated by the
isomorphism classes $[V]$ of $V$ in mod$H$
modulo the relations $[U\oplus V]=[U]+[V]$, $U, V\in{\rm mod}H$. The multiplication of $r(H)$
is determined by $[U][V]=[U\ot V]$, the tensor product of $H$-modules. Then $r(H)$ is an associative ring
with identity.
Notice that $r(H)$ is a free abelian group with a $\mathbb Z$-basis
$\{[V]|V\in{\rm ind}(H)\}$, where ${\rm ind}(H)$ denotes the category
of indecomposable objects in mod$H$.

Note that there is a canonical ring epimorphism $r(H)\ra G_0(H)$, $[V]\mapsto [V]$. If $H$ is a finite dimensional
semisimple Hopf algebra, then the epimorphism is a ring isomorphism, i.e., $r(H)=G_0(H)$.

\subsection{Hopf-Ore extensions of a group algebra $kG$}\selabel{1.2}

Let $G$ be a group and $a\in Z(G)$. Let $\chi\in\hat{G}$ with $\chi(a)\neq1$ and let $q=\chi(a)$.
The Hopf-Ore extension $H=kG(\chi^{-1}, a, 0)$ of the group
algebra $kG$ can be described as follows. $H$ is generated, as an algebra, by $G$ and $x$ subject to
the relations $xg=\chi^{-1}(g)gx$ for all $g\in G$. The coalgebra structure and the antipode are given
by
$$\begin{array}{lll}
\Delta (x)=x\otimes a+1\ot x,& \varepsilon(x)=0,& S(x)=-xa^{-1},\\
\Delta(g)=g\otimes g,& \varepsilon(g)=1,& S(g)=g^{-1},\\
\end{array}$$
where $g\in G$.
$H$ has a $k$-basis $\{gx^i\mid g\in G,i\in\mathbb{N}\}$.

\section{\bf Simple modules}\selabel{3}

In this and the next two sections,  we fix  $H=kG(\chi^{-1}, a, 0)$,  a Hopf-Ore extension of a group algebra $kG$ as defined in the previous section. Let $q=\chi(a)$.

Let $V$ be a $kG$-module. Then $V$ becomes an $H$-module by setting $x\cdot v=0$, $v\in V$ (see \cite[Page 812]{WangYouChen}). Thus, one obtains an embedding functor $F: {\rm mod}kG\ra{\rm mod}H$. Obviously, $F$ is a tensor functor. Hence
mod$kG$ can be regarded as a tensor subcategory of mod$H$.

Let $\{ V_i|i \in I\}$ be all non-isomorphic simple $kG$-module.
For any $i\in I$, $V_i$ becomes a simple $H$-module as above.
For any $\lambda\in\hat{G}$, there is a one-dimensional $H$-module $V_{\lambda}$ defined by $g\cdot v=\l(g)v$ and $x\cdot v=0$ for any $g\in G$ and $v\in V_{\l}$ (see \cite{WangYouChen}). $V_{\l}$ is also a simple $kG$-module. Hence we may regard $\hat{G}\subseteq I$. Thus, $V_{\e}$ is the trivial $H$-module, where $\e$ is the identity of the group $\hat{G}$. One can easily check that $V_i\ot V_{\l}\cong V_{\l}\ot V_i$ is a simple module as well for any $i\in I$ and $\l\in\hat{G}$. Hence there exists a permutation $\s$ of $I$ such that $V_{\chi}\ot V_i\cong V_{\s(i)}$, $i\in I$. Consequently, $V_{\s^t(i)}\cong V_{\chi^t}\ot V_i$, $t\in\mathbb Z$. Define a binary relation $\sim$ on $I$ as follows: $i\sim j$ if $i=\s^t(j)$, or equivalently, $V_i=V_{\sigma^t(j)}$, for some $t\in\mathbb Z$, where $i,j\in I$. Obviously, $\sim$ is an equivalent relation. Denote by $[i]$ the equivalence class containing $i$. Let  $I_0$ be the set of all equivalence classes of $I$ with respect to $\sim$.

Clearly, if $|\chi|=s<\infty$ then $\s^s(i)=i$ for any $i\in I$. Conversely, we have the following lemma.

\begin{lemma}\label{3.1}
If $\s^t(i)=i$ for some $i\in I$ and $t\in\mathbb Z$ with $t\neq 0$, then $|\chi|<\infty$.
\end{lemma}

\begin{proof}
Assume $\s^t(i)=i$ for some $i\in I$ and $t\in\mathbb{N}^+$. Then $V_{\chi^t}\ot V_i\cong V_i$. Let $\phi: V_i\ra V_{\chi^t}\ot V_i$ be a $kG$-module isomorphism. Let $0\neq v_0\in V_{\chi^t}$. Since ${\rm dim}(V_{\chi^t})=1$, there exists a linear automorphism $f$ of $V_i$ such that $\phi(v)=v_0\ot f(v)$, $v\in V_i$. From $\phi(g\cdot v)=g\cdot\phi(v)$, one gets $f(g\cdot v)=\chi^t(g)g\cdot f(v)$, and so $g^{-1}\cdot f(g\cdot v)=\chi^t(g)f(v)$, where $g\in G$ and $v\in V_i$. This implies ${\rm det}(f)=\chi^{t{\rm dim}(V_i)}(g){\rm det}(f)$, $g\in G$. It follows that $\chi^{t{\rm dim}(V_i)}(g)=1$ for all $g\in G$, and so $|\chi|<\infty$.
\end{proof}

Let $\langle\chi\rangle$ be the subgroup of $\hat{G}$ generated by $\chi$ and $\hat{G}/\langle\chi\rangle$
the corresponding quotient group. By \cite[Proposition 3.17(a)]{WangYouChen}, one can see that $\hat{G}/\langle\chi\rangle\subseteq I_0$.

\begin{lemma}\label{3.2}
Let $i\in I$ and $l, r\in\mathbb{Z}$. Then $\s^l(i)\neq\s^r(i)$ if $|\chi|=\infty$ and $l\neq r$, or $|q|=s<\infty$ and $s\nmid l-r$.
\end{lemma}

\begin{proof}
If $|\chi|=\infty$ and $l\neq r$, then $\s^{l-r}(i)\neq i$ by Lemma \ref{3.1}, and hence $\s^l(i)\neq\s^r(i)$. Now assume $|q|=s<\infty$ and $s\nmid l-r$. If $\s^l(i)=\s^r(i)$, then $\s^{l-r}(i)=i$. By the proof of Lemma \ref{3.1}, there is a linear automorphism $f$ of $V_i$ such that $f(av)=\chi^{l-r}(a)af(v)=q^{l-r}af(v)$, $v\in V_i$.
Since $a$ is a central element of $G$ and $V_i$ is a simple $kG$-module, there exists an $\a\in k^{\ti}$ such that $av=\a v$ for all $v\in V_i$. Hence $\a f(v)=q^{l-r}\a f(v)$, $v\in V_i$. This implies $q^{l-r}\a=\a$, and hence $q^{l-r}=1$, a contradiction. This completes the proof.
\end{proof}

For any $H$-module $M$, the subspace $M^x=\{m\in M|xm=0\}$ is a submodule of $M$. If $M^x=M$ then $M$ is called $x$-{\it torsion}. If $M^x=0$ then $M$ is called $x$-{\it torsionfree}. Obviously, if $M$ is a simple $H$-module, then $M$ is either $x$-{\it torsion} or $x$-{\it torsionfree}.

\begin{lemma}\label{3.3}
If there exists a nonzero $x$-{\it torsionfree} $H$-module, then $|\chi|<\infty$.
\end{lemma}

\begin{proof}
Suppose that $M$ is a nonzero $x$-{\it torsionfree} $H$-module. Let $V$ be a simple $kG$-submodule of $M$. Then $V\cong V_i$ for some $i\in I$. Without loss of generality, we may assume that $V_i$ is a simple $kG$-submodule of $M$. Since $M$ is $x$-{\it torsionfree},  $x^jV_i\neq 0$ for any $j\>1$. It is easy to check that $x^jV_i$ is a $kG$-submodule of $M$ and $x^jV_i\cong V_{\s^j(i)}$. Since $M$ is finite dimensional, there is a positive integer $n$ such that $x^nV_i\subseteq\sum_{j=0}^{n-1}x^jV_i$. Since $V_i, xV_i, \cdots, x^nV_i$ are all simple $kG$-modules, $x^nV_i\cong x^lV_i$ as $kG$-modules for some $0\<l\<n-1$. This implies $V_{\s^n(i)}\cong V_{\s^l(i)}$, and hence $\s^n(i)=\s^l(i)$. By Lemma \ref{3.1}, $|\chi|<\infty$.
\end{proof}

Let $i\in I$. Then one can define a module $M(V_i)=H\ot_{kG} V_i$. Note that $H$ is a free right $kG$-module with a basis $\{x^l|l\>0\}$. Hence
$M(V_i)=\oplus_{l=0}^{\infty}x^l\otimes_{kG}V_i$ as vector spaces.
For any $v\in V_i$, still denote by $v$ the element $1\ot_{kG}v$ of $M(V_i)$ for simplicity. Then we may view  $V_i$ as $1\ot_{kG}V_i$  in $M(V_i)$.

Assume $|\chi|=s<\oo$. Let $i\in I$. For any $\b\in k$ and a monic polynomial $f_n(y)=(y-\b)^n=y^n-\sum_{j=0}^{n-1}\a_jy^j$ with $n\> 1$, let $N^n_{\b}(i)$ be the submodule of $M(V_i)$ generated by $f_n(x^s)M(V_i)$, and define $V_n(i,\b)=M(V_i)/N^n_{\b}(i)$ to be the corresponding quotient module. Since $f(x^s)$ is a central element of $H$, $N^n_{\b}(i)=f_n(x^s)M(V_i)$ and hence dim$V_n(i,\b)$=$ns$dim$V_i$. For any $m\in M(V_i)$, denote still by $m$ the image of $m$ under the canonical epimorphism $M(V_i)\ra V_n(i,\b)$. Then it is easy to see that $V_n (i,\b)$ is generated, as an $H$-module, by $V_i$, and $V_n (i,\b)=\oplus_{j=0}^{ns-1}x^jV_i$ as vector spaces. Moreover, we have
$$x^{ns}v=\sum_{j=0}^{n-1}\a_jx^{js}v,\ v\in V_n(i,\b).$$
If $n=1$, then $x^sv=\b v$ for all $v\in V_1(i,\b)$.  Let $V (i,\b)$ be the module  $V_1(i,\b)$.
Obviously, $V(i, \b)$ is $x$-{\it torsionfree} for any $i\in I$ and $\b\in k^{\ti}$.

\begin{proposition}\label{3.4}
Assume $|\chi|=|q|=s<\oo$. Let $i, j\in I$ and $\a, \b\in k^{\ti}$. Then  $V (i,\b)$ is simple and $V (i,\b)\cong V (j,\a)$ if and only if $[i]=[j]$ and $\b=\a$.
\end{proposition}
\begin{proof}
Let $N$ be a nonzero submodule of $V (i,\b)$ and $V$ a simple $kG$-submodule of $N$. Since $V(i,\b)=\oplus_{j=0}^{s-1}x^jV_i$ is $x$-{\it torsionfree}, $x^jV_i$ is a $kG$-submodule of $V(i, \b)$ and $x_jV_i\cong V_{\s^j(i)}$, $0\<j\<s-1$. Then by Lemma \ref{3.2}, $V_i, xV_i, \cdots, x^{s-1}V_i$ are pairwise non-isomorphic simple $kG$-submodules of $V(i,\b)$. It follows that $V=x^jV_i$ for some $0\<j\<s-1$. Thus, $N\supset x^{s-j}V=x^sV_i=\b V_i=V_i$. Since the $H$-module $V(i,\b)$ is generated by $V_i$, $N=V(i,\b)$ and so $V(i,\b)$ is a simple $H$-module.

Assume $V (i,\b)\cong V (j,\a)$. Then dim$V_i$=dim$V_j$. Let $\phi: V (i,\b)\ra V (j,\a)$ be an $H$-module isomorphism. Pick up a nonzero element $v\in V_j$. Then there exist an $m\in V(i,\b)$ such that $\phi(m)=v$. In this case, we have $\phi(x^sm)=\phi(\b m)=\b v$ and $\phi(x^sm)=x^s\phi(m)=x^sv=\a v$. Hence $\b=\a$. Since $\phi$ is an $H$-module isomorphism, $\phi(V_i)$ is a $kG$-submodule of $V (j,\a)$ and $\phi(V_i)\cong V_i$ as $kG$-modules.
By the discussion above, $V(j,\a)=\oplus_{t=0}^{s-1}x^tV_j$ and $V_j, xV_j, \cdots, x^{s-1}V_j$ are non-isomorphic simple $kG$-submodules of $V(j,\a)$. Hence
there exists an integer $t$ with $0\<t\<s-1$ such that $x^tV_j=\phi(V_i)$ since $\phi(V_i)$ is a simple $kG$-submodule of $V(j,\a)$. Therefore, $V_i\cong V_{\sigma^t(j)}$ as $kG$-modules, which implies  $[i]=[j]$.

Conversely, assume that $[i]=[j]$ and $\b=\a$. Then there exists an integer $t$ with $0\<t\<s-1$ such that $V_i=V_{\sigma^t(j)}$. Note that $V_i\subset V(i, \b)$ and $V_j\subset V(j, \b)$ as stated before. Since $x^tV_j$ is a $kG$-submodule of $V(j, \b)$ and $x^tV_j\cong V_{\sigma^t(j)}$, we have $V_i\cong x^tV_j$ as $kG$-modules. Let $\phi: V_i\ra x^tV_j$ be a $kG$-module isomorphism.
Since $V(i, \b)=\oplus_{l=0}^{s-1}x^lV_i$, we may extend $\phi$ to a linear map $\phi_0$ from $V(i, \b)$ to $V(j, \b)$ by letting $\phi_0(x^lv)=x^l\phi(v)$ for all $0\<l\<s-1$ and $v\in V_i$.
It is easy to check that $\phi_0$ is an $H$-module homomorphism.
Since both $V(i, \b)$ and $V(j, \b)$ are simple, it follows from $\phi_0\neq 0$
that $\phi_0$ is an $H$-module isomorphism.
\end{proof}

\begin{theorem}\label{3.5}
Let $M$ be a simple $H$-module.
\begin{enumerate}
\item[(1)] If $M$ is $x$-{\it torsion}, then $M\cong V_i$ for some $i\in I$.
\item[(2)] If $|\chi|=|q|=s$ and $M$ is $x$-{\it torsionfree}, then $s<\infty$ and $M\cong V(i,\b)$ for some $i\in I$ and $\b\in k^{\ti}$.
    \end{enumerate}
\end{theorem}
\begin{proof}
(1) If $M$ is $x$-{\it torsion}, then $M$ is a simple $kG$-submodule. Hence there is an $i\in I$ such that $M\cong V_i$ as $kG$-modules, and so $M\cong V_i$ as $H$-modules.

(2) Assume that $|\chi|=|q|=s$ and $M$ is $x$-{\it torsionfree}. Then $s<\infty$ by Lemma \ref{3.3}. Let $V$ be a simple $kG$-submodule of $M$. Then $V\cong V_i$ for some $i\in I$. Without loss of generality, we may assume that $V_i$ is a simple $kG$-submodule of $M$. Define a liner map $\phi: M(V_i)=H\otimes_{kG}  V_i\rightarrow M$ by $\phi(h\otimes v)=hv$ for any $h\in H$ and $v\in V_i$. Since $M$ is a simple $H$-module, it is easy to see that $\phi$ is an $H$-module epimorphism.
Since $x^s$ is central element in $H$ and $M$ is $x$-{\it torsionfree}, there exists a $\b\in k^{\ti}$ such that $x^sm=\b m$ for any $m\in M$, i.e., $(x^s-\b)M=0$. Hence  $\phi(N_{\b}^1(i))=\phi((x^s-\b)M(V_i))=(x^s-\b)\phi(M(V_i))=(x^s-\b)M=0$. Thus, $\phi$ induces an $H$-module epimorphism $\widetilde{\phi}: V(i, \b)=M(V_i)/N_{\b}^1(i)\ra M$, which must be an isomorphism since $V(i, \b)$ and $M$ are both simple $H$-modules.
\end{proof}

\begin{corollary}\label{3.6}
The following statements hold.
\begin{enumerate}
\item[(1)] If $|\chi|=\oo$, then $\{V_i|i\in I\}$ is a representative set of isomorphic classes of simple $H$-modules.
\item[(2)] If $|\chi|=|q|=s<\oo$, then $\{V_i, V(j,\b)|i\in I,[j]\in I_0, \b\in k^{\ti}\}$ is a representative set of isomorphic classes of simple $H$-modules.
\end{enumerate}
\end{corollary}
\begin{proof}
It follows from Lemmas \ref{3.2}-\ref{3.3}, Proposition \ref{3.4} and Theorem \ref{3.5}.
\end{proof}
\section{\bf Indecomposable modules}\selabel{4}
Throughout this section, we assume that the group algebra $kG$ is finite dimensional and semisimple. We will use the notations of last section and let $|\chi|=s$. In this case, $1<|q|\<s<\infty$. Moreover, $I$ and $I_0$ are finite sets.

Let $A$ be a $k$-algebra, and $M$ an $A$-module. Then the smallest nonnegative integer $l$ with ${\rm rad}^l(M)=0$ is called the {\it radical length} of $M$, denoted by ${\rm rl}(M)$, and $0\subset{\rm rad}^{l-1}(M)\subset\cdots\subset{\rm rad}^2(M)\subset{\rm rad}(M)\subset M$ is called the {\it radical series} of $M$. By \cite[Proposition II.4.7]{ARS}, ${\rm rl}(M)={\rm sl}(M)$, the socle length of $M$, which is sometimes called the {\it Loewy length} of $M$. Let ${\rm l}(M)$ denote the length of $M$.

For any $t\in\mathbb{N}^+$ and $i\in I$, let $J_t(i)$ be the submodule of $M(V_i)$ generated by $x^tV_i$, and define $V_t(i)= M ( V_i)/J_t(i)$ to be the corresponding quotient module. Note that $J_t(i)=x^tM(V_i)=\oplus_{j=t}^{\infty}x^jV_i$. For simplicity, denote still by $z$ the image of an element $z\in M(V_i)$ under the canonical epimorphism $M(V_i)\ra V_t(i)$. Then $V_i\subseteq V_t(i)$, $V_t(i)=\oplus_{j=0}^{t-1}x^jV_i$ as vector spaces, dim$V_t(i)=t {\rm dim}V_i$ and $x^tV_t(i)=0$. Moreover, $x^jV_i$ is a $kG$-submodule of $V_t(i)$ and $x^jV_i\cong V_{\s^j(i)}$, $0\<j\<t-1$.

\begin{remark}\label{4.1}
$V_1(i)\cong V_i$ and $V_t(i,0)=V_{ts}(i)$ as $H$-modules, where $i\in I$ and $t\in\mathbb{N}^+$.

\end{remark}

The following lemma is obvious.

\begin{lemma}\label{4.2}
Let $i\in I$ and $t\in\mathbb{N}^+$. Then for any $0\<j\<t-1$ and $0\neq v\in V_i$, $x^jv\neq 0$ in $V_t(i)$.
\end{lemma}

\begin{proposition}\label{4.3}
Let $i\in I$ and $t\in\mathbb{N}^+$. Then $V_t(i)$ is an indecomposable uniserial $H$-module. Moreover, ${\rm l}(V_t(i))=t$.
\end{proposition}
\begin{proof}
For any $0\<l\<t-1$, let $M_l=\sum_{j=l}^{t-1}x^jV_i\subseteq V_t(i)$. Then $M_j$ is an $H$-submodule of $V_t(i)$. Obviously,
$$V_t(i)=M_0\supset M_1\supset\cdots\supset M_{t-1}\supset M_t=0$$
is a composition series of $V_t(i)$ and $M_l/M_{l+1}\cong V_{\s^l(i)}$, $0\<l\<t-1$. Hence ${\rm l}(V_t(i))=t$.

%If $s\>t$ then $V_i, xV_i, \cdots, x^{t-1}V_i$ are non-isomorphic simple $kG$-submodules of $V_t(i)$. Hence $N$ is the sum of those $x^jV_i$ %with $x^jV_i\subseteq N$. Let $l={\rm min}\{j|0\<j\<t-1, x^jV_i\subseteq N\}$.
%Then obviously, $N=\sum_{j=l}^{t-1}x^jV_i=M_l$. Hence $V_t(i)$ is uniserial, and so it is indecomposable.

%Now assume $s<t$.

Let $N$ be a nonzero submodule of $V_t(i)$. Since $x^tN\subseteq x^tV_t(i)=0$, there is an integer $l$ with $1\<l\<t$ such that $x^lN=0$ but $x^{l-1}N\neq 0$. If $l=t$ then $N\subseteq V_t(i)=M_0=M_{t-l}$. If $l<t$ and $z\in N$, then $z=\sum_{j=0}^{t-1}x^jv_j$ for some $v_j\in V_i$. Since $x^lz=\sum_{j=0}^{t-1}x^{l+j}v_j=\sum_{j=0}^{t-1-l}x^{l+j}v_j=0$,
$x^{l+j}v_j=0$ for any $0\<j\<t-1-l$. By Lemma \ref{4.2}, $v_j=0$ for any $0\<j\<t-1-l$. Hence $z=\sum_{j=t-l}^{t-1}x^jv_j\in M_{t-l}$, and so $N\subseteq M_{t-l}$. Thus, we have proven $N\subseteq M_{t-l}$. Since $x^{l-1}N\neq 0$, we may choose an element $z\in N$ such that $x^{l-1}z\neq 0$. From $N\subseteq M_{t-l}$, we have $z=\sum_{j=t-l}^{t-1}x^jv_j$ for some $v_j\in V_i$. Hence $0\neq x^{l-1}z=x^{t-1}v_{t-l}\in N\cap(x^{t-1}V_i)$. This implies $v_{t-l}\neq 0$ and  $x^{t-1}V_i\subseteq N$ since $x^{t-1}V_i$ is simple as a $kG$-module. Now suppose that $1\<r<l$ and $x^jV_i\subseteq N$ for all $t-r\<j\<t-1$. Then $x^{l-r-1}z=\sum_{j=t-l}^{t-1}x^{l-r-1+j}v_j=\sum_{j=t-l}^{t-l+r}x^{l-r-1+j}v_j\in N$. Hence $0\neq  x^{t-r-1}v_{t-l}=x^{l-r-1}z-\sum_{j=t-l+1}^{t-l+r}x^{l-r-1+j}v_j\in N\cap(x^{t-(r+1)}V_i)$, and  so  $x^{t-(r+1)}V_i\subseteq N$.
Thus, we have shown that $x^jV_i\subseteq N$ for all $t-l\<j\<t-1$. Therefore, $M_{t-l}\subseteq N$, and so $N=M_{t-l}$. It follows that $V_t(i)$ is uniserial and indecomposable.
\end{proof}

\begin{corollary}\label{4.4}
Let $i, j\in I$ and $n, t\in\mathbb{N}^+$. Then $V_t(i)\cong V_n(j)$ if and only if $t=n$ and $i=j$.
\end{corollary}
\begin{proof}
By Proposition \ref{4.3} and its proof,  we have ${\rm l}(V_t(i))=t$, ${\rm l}(V_n(j))=n$,  $V_t(i)/{\rm  rad}(V_t(i))\cong V_i$ and $V_n(j)/{\rm rad}(V_n(j))\cong V_j$. Hence the corollary follows.
\end{proof}

\begin{lemma}\label{4.5}
Let $M$ be an $H$-module. If each composition factor of $M$ is isomorphic  to $V_i$ for some  $i\in I$, then $xM={\rm rad}(M)$ and $M^x={\rm soc}(M)$.
\end{lemma}

\begin{proof}
Assume that each composition factor of $M$ is isomorphic to some $V_i$. Then $x(M/{\rm rad}(M))=0$, and hence $xM\subseteq{\rm rad}(M)$. On the other hand, it is easy to see that $xM$ is a submodule of $M$. Let $\overline{M}=M/xM$. Then $x\overline{M}=0$, and hence each $kG$-submodule of $\overline{M}$ is an $H$-submodule of $\overline{M}$. So $\overline{M}$ is a semisimple $H$-module, which implies that ${\rm rad}(M)\subseteq xM$. Therefore $xM={\rm rad}(M)$. Similarly, from $xM^x=0$, one gets $M^x\subseteq{\rm soc}(M)$. By the assumption on $M$, one knows that each simple submodule of $M$ is contained in $M^x$. Hence ${\rm soc}(M)\subseteq M^x$, and so $M^x={\rm soc}(M)$.
\end{proof}

\begin{theorem}\label{4.6}
Let $M$ be an indecomposable $H$-module. If each composition factor of $M$ is isomorphic  to $V_j$ for some $j\in I$. Then $M$ is isomorphic to some $V_t(i)$, where $i \in I$ and $t\in\mathbb{N}^+$.
\end{theorem}
\begin{proof}
Assume that each composition factor of $M$ is isomorphic to some $V_j$.
Define a linear endomorphism $\phi$ of $M$ by $\phi(m)=xm$, $m\in M$. Then using the map $\phi$,  it follows from Lemma \ref{4.5} and \cite[Lemma 4.1]{WangYouChen} that $M$ is uniserial. Hence the radical series of $M$ is its unique composition series. Let $t={\rm l}(M)$. Then $t\>1$ and $x^tM=0$ but $x^{t-1}M\neq 0$ by Lemma \ref{4.5}. Since $M$ is semisimple as a $kG$-module, there is a simple $kG$-submodule $V$ of $M$ such that $x^{t-1}V\neq 0$. Let $N=\sum_{j=0}^{t-1}x^jV$. From $x^tV\subseteq x^tM=0$, it is easy to see that $N$ is an $H$-submodule of $M$, and consequently $N$ is also uniserial. Hence ${\rm l}(N)$ is equal to the radical length of $N$. Clearly, $x^{t-1}N=x^{t-1}V\neq 0$ and  $x^tN=0$. Thus, by Lemma \ref{4.5}, one knows that ${\rm l}(N)=t={\rm l}(M)$, which implies $M=N=\sum_{j=0}^{t-1}x^jV$. Since $V$ is a simple $kG$-submodule of $M$, there exists an $i\in I$ such that $V\cong V_i$ as $kG$-modules. Let $f: V_i\ra V$ be a $kG$-module isomorphism. Define a linear map $\psi: M(V_i)\ra M$ to be the composition
$$\psi: M(V_i)=H\ot_{kG}V_i\xrightarrow{{\rm id}\ot f}H\ot_{kG}V\xrightarrow{\cdot}M.$$
That is, $\psi(h\ot v)=hf(v)$ for any $h\in H$ and $v\in V_i$. Obviously, $\psi$ is an $H$-module epimorphism. Now we have $\psi(J_t(i))=\psi(x^tM(V_i))=x^t\psi(M(V_i))=x^tM=0$. Hence $\psi$ induces an $H$-module epimorphism $\ol{\psi}: V_t(i)=M(V_i)/J_t(i)\ra M$. Since ${\rm l}(V_t(i))=t={\rm l}(M)$, $\ol{\psi}$ is an $H$-module isomorphism.
This completes the proof.
\end{proof}

%\begin{corollary}\label{4.5}
%Assume that $|\chi|=|q|=\infty$. Let $M$ be an indecomposable $H$-module. Then $M$ %is isomorphic to some $V_t(i)$, where $t\>1$ and $i\in I$. Moreover, %$\{V_t(i)|t\>1, i\in I\}$ is a representative set of isomorphic classes of finite %dimensional indecomposable $H$-modules.
%\end{corollary}
%\begin{proof}
%It follows form Corollary \ref{3.6}, Corollary \ref{4.4} and Theorem \ref{4.6}.
%\end{proof}

Let $M$ be an arbitrary $H$-module. For any monic polynomial $f(y)\in k[y]$, put
$$M^{(f)}=\{m\in M|f(x^s)^rm=0 \text{ for some integer }r>0\}.$$
Note that \cite[Lemma 4.10, Theorem 4.11, Corollary 4.12, Lemma 4.13 ]{WangYouChen} still hold. When $f(y)=y-\b$ for some $\b\in k$, we denote $M^{(f)}$ by $M^{(\b)}$.

In the rest of this section, assume $|q|=|\chi|=s$.

\begin{lemma}\label{4.7}
Let $M$ be an indecomposable $H$-module. If there exists a scalar $\b\in k^{\ti}$ such that $(x^s-\b)M=0$, then $M$ is simple and isomorphic to $V(i,\b)$ for some $i\in I$.
\end{lemma}
\begin{proof}
Clearly, $M^x=0$. Let $U_1$ be a simple $kG$-submodule of $M$ and $M_1=HU_1$ be the $H$-submodule of $M$ generated by $U_1$. Then there is an $i_1\in I$ such that $U_1\cong V_{i_1}$ as $kG$-modules. Let $f_1: V_{i_1}\ra U_1$ be a $kG$-module isomorphism. Then the composition map
$$\phi_1: M(V_{i_1})=H\ot_{kG}V_{i_1}\xrightarrow{{\rm id}\ot f_1}H\ot_{kG}U_1\xrightarrow{\cdot}M_1,\ h\ot v\mapsto hf_1(v)$$
is an $H$-module epimorphism. Since $(x^s-\b)M=0$, $\phi_1(N^1_{\b}(i_1))=\phi_1((x^s-\b)M(V_{i_1}))=(x^s-\b)\phi_1(M(V_{i_1}))=(x^s-\b)M_1=0$.
Hence $\phi_1$ induces an $H$-module epimorphism $\ol{\phi_1}: V(i_1, \b)=M(V_{i_1})/N^1_{\b}(i_1)\ra M_1$, which must be an isomorphism since $V(i_1, \b)$ is a simple $H$-module. Now let $l\>1$ and suppose that we have found simple $H$-submodules $M_1, \cdots, M_l$ of $M$ such that the sum $\sum_{j=1}^lM_j$ in $M$ is direct and $M_j\cong V(i_j, \b)$ for some $i_j\in I$, $1\<j\<l$. If $\sum_{j=1}^lM_j\neq M$, then there is a simple $kG$-submodule $U_{l+1}$ of $M$ such that $U_{l+1}\nsubseteqq\sum_{j=1}^lM_j$. Let $M_{l+1}=HU_{l+1}$ be the $H$-submodule of $M$ generated by $U_{l+1}$. Then a similar argument as above shows that $M_{l+1}\cong V(i_{l+1}, \b)$ for some $i_{l+1}\in I$. Thus, $M_{l+1}$ is simple and $M_{l+1}\nsubseteqq\sum_{j=1}^lM_j$. Hence the sum $\sum_{j=1}^{l+1}M_j$ in $M$ is direct. Since $M$ is finite dimensional, there are finitely many simple $H$-submodules $M_1, \cdots, M_m$ of $M$ such that $M=\oplus_{j=1}^mM_j$ and $M_j\cong V(i_j, \b)$ for some $i_j\in I$, $1\<j\<m$. Since $M$ is indecomposable, $m=1$. This completes the proof.
\end{proof}

\begin{lemma}\label{4.8}
Assume $\b\in k^{\ti}$. Let $M$ be an indecomposable $H$-module with $M=M^{(\b)}$. Then each composition factor of $M$ is isomorphic to $V(i,\b)$ for some $i\in I$.
\end{lemma}
\begin{proof}
Let $N$ be the composition factor of $M$. Then $N=N^{(\b)}$ by $M=M^{(\b)}$. By \cite[Lemma 4.13 ]{WangYouChen}, one knows that $(x^{s}-\b)N=0$. Then it follows from Lemma \ref{4.7} that $N\cong V(i,\b)$ for some $i\in I$.
\end{proof}

\begin{lemma}\label{4.9}
Assume $\b\in k^{\ti}$. Let $M$ be an $H$-module with $M=M^{(\b)}$. Then ${\rm rad}(M)=(x^{s}-\b)M$.
\end{lemma}
\begin{proof}
Since $M/{\rm rad}(M)$ is semisimple, $(x^{s}-\b)(M/{\rm rad}(M))=0$ by Lemma \ref{4.8}. Hence $(x^s-\b)M\subseteq{\rm rad}(M)$. On the other hand, we have $(x^s-\b)(M/(x^s-\b)M)=0$. Hence it follows from the proof of Lemma \ref{4.7} that $M/(x^s-\b)M$ is semisimple. This implies ${\rm rad}(M)\subseteq (x^s-\b)M$, and so ${\rm rad}(M)=(x^{s}-\b)M$.
\end{proof}

\begin{proposition}\label{4.10}
Let $i\in I$, $\b\in k^{\ti}$ and $r\in\mathbb{N}^+$. Then $V_r(i,\b)$ is uniserial and indecomposable. Moreover, ${\rm l}(V_r(i,\b))=r$ and the composition factors of $V_r(i,\b)$ are all isomorphic to $V(i,\b)$.
\end{proposition}
\begin{proof}
Since $(x^s-\b)^rV_r(i,\b)=0$, $V_r(i,\b)=V_r(i,\b)^{(\b)}$. It follows from Lemma \ref{4.9} that ${\rm rad}^j(V_r(i,\b))=(x^s-\b)^jV_r(i,\b)$ for any $j\>0$. Clearly, $(x^s-\b)^{r-1}V_r(i,\b)\neq 0$. Hence the series
$$0\subset(x^s-\b)^{r-1}V_r(i,\b)\subset(x^s-\b)^{r-2}V_r(i,\b)\subset\cdots
\subset(x^s-\b)V_r(i,\b)\subset V_r(i,\b)$$
is the radical series of $V_r(i, \b)$, and so ${\rm rl}(V_r(i, \b))=r$. Let $\pi: M(V_i)\ra V_r(i, \b)$ be the canonical $H$-module epimorphism. Let $0\<j\<r-1$. Since $x^s-\b$ is a central element of $H$, the map $\psi: V_r(i, \b)\ra(x^s-\b)^jV_r(i, \b), \ v\mapsto (x^s-\b)^jv$ is an $H$-module epimorphism.
Hence the composition map $\phi=\psi\circ\pi$ is an $H$-module epimorphism from $M(V_i)$ to $(x^s-\b)^jV_r(i, \b)$. Since $\phi(N_{\b}^1(i))=\phi((x^s-\b)M(V_i))=(x^s-\b)\phi(M(V_i))=(x^s-\b)^{j+1}V_r(i, \b)$, $\phi$ induces an $H$-module epimorphism $\ol{\phi}: V(i, \b)=M(V_i)/N_{\b}^1(i)\ra (x^s-\b)^{j}V_r(i, \b)/(x^s-\b)^{j+1}V_r(i, \b)$.
By Proposition \ref{3.4}, $V(i,\b)$ is simple. Hence $\ol{\phi}$ must be an isomorphism. Thus, the above radical series of $V_r(i, \b)$ is a composition series. It follows that $V_r(i,\b)$ is uniserial and indecomposable. Moreover, ${\rm l}(V_r(i, \b))={\rm rl}(V_r(i, \b))=r$ and each composition factor of $V_r(i, \b)$ is isomorphic to $V(i, \b)$.
\end{proof}

\begin{theorem}\label{4.11}
Let $M$ be an indecomposable $H$-module. Then $M\cong V_t(i)$ for some $i\in I$ and $t\in\mathbb{N}^+$, or $M\cong V_r(i,\b)$ for some $i\in I$, $\b\in k^{\ti}$ and $r\in\mathbb{N}^+$. Moreover, $M$ is uniserial.
\end{theorem}

\begin{proof}
By \cite[Corollary 4.12]{WangYouChen}, there exists a monic irreducible polynomial $f(y)\in k[y]$ such that $M=M^{(f)}$. Since $k$ is an algebraically closed field, $f(y)=y$ or $f(y)=y-\b$ for some $\b\in k^{\ti}$.

Case 1: $f(y)=y$. Since $M$ is finite dimensional, $x^{rs}M=0$ for some integer $r\>1$. Let $V$ be a composition factor of $M$. Then $x^{rs}V=0$, and hence $V^x\neq 0$. Since $V$ is simple, $V^x=V$. By Theorem \ref{3.5}, $V\cong V_i$ for some $i\in I$. It follows from Theorem \ref{4.6} that $M\cong V_t(i)$ for some integer $t\>1$ and $i\in I$. In this case, $M$ is uniserial by Proposition \ref{4.3}.

Case 2: $f(y)=y-\b$. In this case, $M=M^{(\b)}$. It follows from Lemma \ref{4.8} that each composition factor of $M$ is isomorphic to $V(i,\b)$ for some $i\in I$.
If $rl(M)=1$, then $M$ is simple, and so $M\cong V(i,\b)$ for some $i\in I$.

Now assume ${\rm rl}(M)=r>1$. Then $(x^s-\b)^rM=0$ and $(x^s-\b)^{r-1}M\neq 0$ by Lemma \ref{4.9}. Define a liner map $\phi: M\rightarrow M$ by $\phi(m)=(x^s-\b)m$, $m\in M$. Then $\phi$ is a module endomorphism of $M$ since $x^s-\b$ is a central element of $H$. For any submodule $N$ of $M$, $\phi(N)={\rm rad}(N)$ by Lemma \ref{4.9}, and $\phi^{-1}(N)$ is obviously a submodule of $M$.
If $V$ is a simple submodule of $M$, then $(x^s-\b)V={\rm rad}(V)=0$ by Lemma \ref{4.9}, and hence $V\subseteq{\rm Ker}(\phi)$. Thus, ${\rm soc}(M)\subseteq{\rm Ker}(\phi)$. On the other hand, by Lemma \ref{4.7}, ${\rm Ker}(\phi)$ is semisimple, and hence ${\rm Ker}(\phi)\subseteq{\rm soc}(M)$. Therefore, ${\rm Ker}(\phi)={\rm soc}(M)$. It follows from \cite[Lemma 4.1(c)]{WangYouChen} that $M$ is uniserial. Hence ${\rm l}(M)={\rm rl}(M)=r$. Since $M$ is semisimple as a $kG$-module, $M$ is equal to a direct sum of some simple $kG$-submodules of $M$. Then from $(x^s-\b)^{r-1}M\neq 0$, one knows that there is a simple $kG$-submodule $V$ such that $(x^s-\b)^{r-1}V\neq 0$. From $(x^s-\b)^rM=0$, one gets $(x^s-\b)^rV=0$. Let $N=HV$ be the $H$-submodule of $M$ generated by $V$. Then $(x^s-\b)^{r-1}N\neq 0$ and $(x^s-\b)^rN=H(x^s-\b)^rV=0$. By Lemma \ref{4.9}, ${\rm rl}(N)=r$. Since $M$ is uniserial, so is $N$. Hence ${\rm l}(N)={\rm rl}(N)=r={\rm l}(M)$, and so $M=N=HV$. Since $V$ is a simple $kG$-module, $V\cong V_i$ as $kG$-modules for some $i\in I$. Let $f: V_{i}\ra V$ be a $kG$-module isomorphism. Then one gets an $H$-module epimorphism
$$\phi: M(V_{i})=H\ot_{kG}V_{i}\xrightarrow{{\rm id}\ot f}H\ot_{kG}V\xrightarrow{\cdot}M,\ h\ot v\mapsto hf(v).$$
Since $\phi((x^s-\b)^rM(V_i))=(x^s-\b)^rM=0$, $\phi$ induces an $H$-module epimorphism $\ol{\phi}$ from $V_r(i, \b)=M(V_i)/N_{\b}^r(i)$ to $M$. By Proposition \ref{4.10}, ${\rm l}(V_r(i, \b))=r={\rm l}(M)$, and hence $\ol{\phi}$ is an isomorphism.
\end{proof}

\begin{proposition}\label{4.12}
Let $i, j\in I$, $\a, \b\in k^{\ti}$ and $r, t\in\mathbb{N}^+$. Then $V_r(i,\a)\cong V_t(j,\b)$ if and only if $r=t$, $\a=\b$ and $[i]=[j]$.
\end{proposition}
\begin{proof}
If $V_r(i,\a)\cong V_t(j,\b)$, then $r={\rm l}(V_r(i, \a))={\rm l}(V_t(j, \b))=t$ and $V(i,\a)\cong V(j,\b)$ by Proposition \ref{4.10}, and consequently $\a=\b$ and $[i]=[j]$ by Proposition \ref{3.4}. Conversely, assume that $r=t$, $\a=\b$ and $[i]=[j]$. We need to show $V_t(i,\b)\cong V_t(j,\b)$. By $[i]=[j]$, $i=\s^n(j)$ for some $0\<n\<s-1$. Note that $V_t(j, \b)$ is generated, as an $H$-module, by $V_j$, and $V_t(j, \b)=\oplus_{l=0}^{ts-1}x^lV_j$. From $(x^s-\b)^tV_t(j, \b)=0$, one gets $V_t(j, \b)^x=0$. Hence $x^nV_j$ is a nonzero $kG$-submodule of $V_t(j, \b)$ and $x^nV_j\cong V_{\s^n(j)}=V_i$. Let $M$ be the $H$-submodule of $V_t(j, \b)$ generated by $x^nV_j$. Then $M=H(x^nV_j)=x^nHV_j=x^nV_t(j, \b)=V_t(j, \b)$ by $V_t(j, \b)^x=0$. Let $f: V_i\ra x^nV_j$ be a $kG$-module isomorphism. Then an argument similar to the proof of Theorem \ref{4.11} shows that $f$ can be extended to an $H$-module isomorphism from $V_t(i, \b)$ to $M=V_t(j, \b)$.
\end{proof}

\begin{corollary}\label{4.13}
Assume that $|q|=|\chi|=s$. Then
$$\{V_t(i),V_t(j,\b)|i \in I,[j]\in I_0,\b\in k^{\ti}, t\in\mathbb{N}^+\}$$ is a representative set of isomorphic classes of finite dimensional indecomposable $H$-modules.
\end{corollary}

\begin{proof}
It follows from Proposition \ref{4.3}, Corollary \ref{4.4}, Proposition \ref{4.10}, Theorem \ref{4.11} and Proposition \ref{4.12}.
\end{proof}

\section{\bf Decomposition rules for tensor product modules}\selabel{5}

Throughout this section, assume that $k$ is of characteristic zero and $G$ is a finite group.
We also assume $|q|=|\chi|=s$. In this case, the group algebra $kG$ is finite dimensional and semisimple,
and  $1<s<\infty$.
In this section, we investigate the decomposition rules for tensor product modules over $H$.

By Corollary \ref{4.13}, one knows that
$$\{V_t(i),V_t(j,\b)|i \in I,[j]\in I_0,\b\in k^{\ti}, t\in\mathbb{N}^+\}$$
is a representative set of isomorphic classes of finite dimensional indecomposable $H$-modules.

As stated in \seref{3}, mod$kG$ is a tensor subcategory of mod$H$.

Recall from \cite{WangYouChen} that an $H$-module $M$ is a {\it weight module} if $M=\oplus_{\l\in\hat{G}}M_{(\l)}$, where $M_{(\l)}=\{m\in M|gm=\l(g)m, \forall g\in G\}$ for any $\l\in\hat{G}$. Let wmod$H$ be the full subcategory of mod$H$ consisting of all finite dimensional weight $H$-modules. Then wmod$H$ is a tensor subcategory of mod$H$ \cite{WangYouChen}. By \cite[Corollary 4.20]{WangYouChen}, $$\{V_t(\l),V_t(\theta,\b)|\l\in\hat{G}, [\theta]\in\hat{G}/\langle\chi\rangle,\b\in k^{\ti}, t\in\mathbb{N}^+\}$$
is a representative set of isomorphic classes of finite dimensional indecomposable weight modules over $H$.

\begin{convention}\label{5.1}
For any $i\in I$, there is a scalar $\om_i\in k^{\ti}$ such that $av=\om_iv$ for all $v\in V_i$ since $a\in Z(G)$ and $V_i$ is a simple $kG$-module. When $i=\l\in\hat{G}$, $\om_{\l}=\l(a)$.

Let $N_{i,j}^l=[V_i\ot V_j:V_l]\in\mathbb N$ be the multiplicity of $V_l$ in a composition series of $V_i\ot V_j$, $i,j, l\in I$. Then $N_{j,i}^l=N_{i,j}^l$ and $V_i\ot V_j\cong\oplus_{l\in I}N_{i,j}^lV_l$ in {\rm mod}$kG$ (or equivalently, in {\rm mod}$H$) since $kG$ is semisimple.
\end{convention}

\begin{lemma}\label{5.2}
Let $i\in I$, $t\in \mathbb{N}^+$ and $\b \in k^{\ti}$. Then
\begin{enumerate}
\item[(1)] $V_i\otimes V_t(\e)\cong V_t(\e)\otimes V_i\cong V_t(i)$;
\item[(2)] $V_i\otimes V_t(\e,\b)\cong V_t(i,\b)$;
\item[(3)] $V_t(\e,\b)\otimes V_i\cong V_t(i,\om^s_i\b)$.
\end{enumerate}
\end{lemma}
\begin{proof}
The proofs of the three isomorphisms are similar. We only prove (3).
From \seref{3}, one knows that $V_t(\e,\b)=\oplus_{j=0}^{ts-1}x^jV_{\e}$
and $V_t(i,\om_i^s\b)=\oplus_{j=0}^{ts-1}x^jV_i$ as $kG$-modules.
Let $0\neq v_0\in V_{\e}$ and $v_j=x^jv_0$, $1\<j\<ts-1$. Then $\{v_0, v_1, \cdots, v_{ts-1}\}$ is a $k$-basis of $V_t(\e,\b)$. Hence one can define a $k$-linear isomorphism
$\phi: V_t(\e,\b)\otimes V_i\ra V_t(i,\om^s_i\b)$ by $\phi(v_j\ot v)=\om_i^{-j}x^jv$ for all $0\<j\<ts-1$ and $v\in V_i$. It is easy to check that $\phi(g(v_j\ot v))=g\phi(v_j\ot v)$ for all $g\in G$, $0\<j\<ts-1$ and $v\in V_i$. Now let $0\<j<ts-1$ and $v\in V_i$. Then
$\phi(x(v_j\ot v))=\phi(xv_j\ot av)=\om_i\phi(v_{j+1}\ot v)=\om_i\om_i^{-(j+1)}x^{j+1}v=x(\om_i^{-j}x^jv)=x\phi(v_j\ot v)$.
Let $(y-\b)^t=y^t-\sum_{l=0}^{t-1}\a_ly^l$. Then $(y-\om_i^s\b)^t=y^t-\sum_{l=0}^{t-1}\om_i^{s(t-l)}\a_ly^l$. Hence we have
$\phi(x(v_{ts-1}\ot v))=\phi(xv_{ts-1}\ot av)=\om_i\phi(x^{ts}v_0\ot v)=\om_i\phi(\sum_{l=0}^{t-1}\a_lx^{ls}v_0\ot v)=\om_i\sum_{l=0}^{t-1}\a_l\phi(v_{ls}\ot v)=\sum_{l=0}^{t-1}\a_l\om_i^{1-ls}x^{ls}v$ and $x\phi(v_{ts-1}\ot v)=x(\om_i^{1-ts}x^{ts-1}v)=\om_i^{1-ts}x^{ts}v=\om_i^{1-ts}\sum_{l=0}^{t-1}\om_i^{s(t-l)}\a_lx^{ls}v
=\sum_{l=0}^{t-1}\om_i^{1-sl}\a_lx^{ls}v$. This shows that $\phi(x(v_{ts-1}\ot v))=x\phi(v_{ts-1}\ot v)$, and so $\phi$ is an $H$-module isomorphism.
\end{proof}

\begin{proposition}\label{5.3}
Let $p,t\in \mathbb{N}^+$, $i,j\in I$ and $\b \in k^{\ti}$. Let $p=us+r$ with $u\>0$ and $0\<r<s$. Then
$$\begin{array}{rcl}
V_p(i)\ot V_t(j,\b)
&\cong&(\oplus_{l\in I}\oplus_{1\<m\<\rm min(t,u)}N_{i,j}^l(s-r)V_{2m-1+|t-u|}(l,\b))\\
&&\oplus(\oplus_{l\in I}\oplus_{m=1}^{\rm{min}(t,u+1)}N_{i,j}^lrV_{2m-1+|t-u-1|}(l,\b)),\\
V_t(j,\b)\ot V_p(i)
&\cong&(\oplus_{l\in I}\oplus_{1\<m\<\rm{min}(u,t)}N_{j,i}^l(s-r)V_{2m-1+|t-u|}(l,\om^{s}_i\b))\\
&&\oplus(\oplus_l\oplus_{m=1}^{\rm{min}(u+1,t)}N_{j,i}^lrV_{2m-1+|t-u-1|}(l,\om_i^{s}\b)).\\
\end{array}$$
\end{proposition}
\begin{proof}
By Convention \ref{5.1}, Lemma \ref{5.2} and \cite[Theorem 3.6 ]{H.H}, we have
$$\begin{array}{rl}
V_p(i)\ot V_t(j,\b)\cong&V_i\ot V_p(\e)\ot V_j\ot V_t(\e, \b)\\
\cong&V_i\ot V_j\ot V_p(\e)\ot V_t(\e, \b)\\
\cong&(\oplus_{l\in I}N_{i,j}^lV_l)\ot((\oplus_{1\<m\<\rm min(t,u)}(s-r)V_{2m-1+|t-u|}(\e,\b))\\
&\oplus(\oplus_{m=1}^{\rm{min}(t,u+1)}rV_{2m-1+|t-u-1|}(\e,\b)))\\
\cong&(\oplus_{l\in I}\oplus_{1\<m\<\rm min(t,u)}N_{i,j}^l(s-r)V_{2m-1+|t-u|}(l,\b))\\
&\oplus(\oplus_{l\in I}\oplus_{m=1}^{\rm{min}(t,u+1)}N_{i,j}^lrV_{2m-1+|t-u-1|}(l,\b)).\\
\end{array}$$
Similarly, one can show the second isomorphism.
\end{proof}

\begin{proposition}\label{5.4}
Let $p,t\in \mathbb{N}^+$, $i,j\in I$ and $\b \in k^{\ti}$. Then
$$V_p(i,\alpha)\ot V_t(j,\b)\cong\oplus_{l\in I}\oplus_{m=0}^{s-1}
\oplus_{u=1}^{{\rm min}\{p,t\}}N_{i,j}^l V_{2u-1+|p-t|}(\sigma^m(l),\om^s_j\a+\beta).$$
Moreover, $V_{2u-1+|p-t|}(\sigma^m(l),\om^s_j\a+\beta)\cong V_{s(2u-1+|p-t|)}(\sigma^m(l))$ when $\om^s_j\a+\beta=0$ and $V_{2u-1+|p-t|}(\sigma^m(l),\om^s_j\a+\beta)\cong V_{2u-1+|p-t|}(l,\om^s_j\a+\beta)$ when $\om^s_j\a+\beta\neq 0$.
\end{proposition}

\begin{proof}
The first assertion follows from Convention \ref{5.1}, Lemma \ref{5.2}, \cite[Theorem 3.7 ]{H.H}
and an argument similar to the proof of Proposition \ref{5.3}. The second assertion follows from Remark \ref{4.1} and Proposition \ref{4.12}.
\end{proof}

\begin{proposition}\label{5.5}
 Let  $i,j\in I$, $n,t\in \mathbb{Z}$ with $n\>t\>1$.
 Assume that $n=r's+p'$ and $t=rs+p$ with $0\<p', p\<s-1$.\\
{\rm (1)} Suppose that $p+p'\<s$. If $p\<p'$ then
$$\begin{array}{rl}
&V_n(i)\ot V_t(j)\cong V_t(j)\ot V_n(i)\\
\cong&(\oplus_{l\in I}\oplus_{m=0}^{r}\oplus_{0\<u\<p-1}N_{i,j}^lV_{n+t-1-2ms-2u}(\sigma^{u}(l)))\\
&\oplus(\oplus_{l\in I}\oplus_{0\<m\<r-1}\oplus_{p\<u\<p'-1}N_{i,j}^lV_{(r+r'-2m)s}(\sigma^{u}(l)))\\
&\oplus(\oplus_{l\in I}\oplus_{0\<m\<r-1}\oplus_{p'\<u\<p+p'-1}N_{i,j}^lV_{n+t-1-2ms-2u}(\sigma^{u}(l)))\\
&\oplus(\oplus_{l\in I}\oplus_{0\<m\<r-1}\oplus_{p+p'\<u\<s-1}N_{i,j}^lV_{(r+r'-1-2m)s}(\sigma^{u}(l))),\\
\end{array}$$
and if $p\>p'$ then
$$\begin{array}{rl}
&V_n(i)\ot V_t(j)\cong V_t(j)\ot V_n(i)\\
\cong&(\oplus_{l\in I}\oplus_{m=0}^{r}\oplus_{0\<u\<p'-1}N_{i,j}^lV_{n+t-1-2ms-2u}(\sigma^{u}(l)))\\
&\oplus(\oplus_{l\in I}\oplus_{m=0}^{r}\oplus_{p'\<u\<p-1}N_{i,j}^lV_{(r+r'-2m)s}(\sigma^{u}(l)))\\
&\oplus(\oplus_{l\in I}\oplus_{0\<m\<r-1}\oplus_{p\<u\<p+p'-1}N_{i,j}^lV_{n+t-1-2ms-2u}(\sigma^{u}(l)))\\
&\oplus(\oplus_{l\in I}\oplus_{0\<m\<r-1}\oplus_{p+p'\<u\<s-1}N_{i,j}^lV_{(r+r'-1-2m)s}(\sigma^{u}(l))).\\
\end{array}$$
{\rm (2)} Suppose that $p+p'\>s+1$ and let $\ol{m}=p+p'-s-1$. If $p\<p'$ then
$$\begin{array}{rl}
&V_n(i)\ot V_t(j)\cong V_t(j)\ot V_n(i)\\
\cong&(\oplus_{l\in I}\oplus_{m=0}^{r}\oplus_{u=0}^{\ol m}N_{i,j}^lV_{(r+r'+1-2m)s}(\sigma^{u}(l)))\\
&\oplus(\oplus_{l\in I}\oplus_{m=0}^{r}\oplus_{u=\ol{m}+1}^{p-1}N_{i,j}^lV_{n+t-1-2ms-2u}(\sigma^{u}(l)))\\
&\oplus(\oplus_{l\in I}\oplus_{0\<m\<r-1}\oplus_{p\<u\<p'-1}N_{i,j}^lV_{(r+r'-2m)s}(\sigma^{u}(l)))\\
&\oplus(\oplus_{l\in I}\oplus_{0\<m\<r-1}\oplus_{u=p'}^{s-1}N_{i,j}^lV_{n+t-1-2ms-2u}(\sigma^{u}(l))),\\
\end{array}$$
and if $p\>p'$ then
$$\begin{array}{rl}
&V_n(i)\ot V_t(j)\cong V_t(j)\ot V_n(i)\\
\cong&(\oplus_{l\in I}\oplus_{m=0}^{r}\oplus_{u=0}^{\ol m}N_{i,j}^lV_{(r+r'+1-2m)s}(\sigma^{u}(l)))\\
&\oplus(\oplus_{l\in I}\oplus_{m=0}^{r}\oplus_{u={\ol m}+1}^{p'-1}N_{i,j}^lV_{n+t-1-2ms-2u}(\sigma^{u}(l)))\\
&\oplus(\oplus_{l\in I}\oplus_{m=0}^{r}\oplus_{p'\<u\<p-1}N_{i,j}^lV_{(r+r'-2m)s}(\sigma^{u}(l)))\\
&\oplus(\oplus_{l\in I}\oplus_{0\<m\<r-1}\oplus_{u=l}^{s-1}N_{i,j}^lV_{n+t-1-2ms-2u}(\sigma^{u}(l))).\\
\end{array}$$
\end{proposition}
\begin{proof}
It follows from Convention \ref{5.1}, Lemma \ref{5.2}, \cite[Theorem 3.15 ]{H.H}
and an argument similar to the proof of Proposition \ref{5.3}.
\end{proof}

\begin{remark}\label{5.6}
Let $r_w(H)$ denote the Green ring of {\rm wmod}$H$.
Since {\rm mod}$kG$ and {\rm wmod}$H$ are both tensor subcategories of {\rm mod}$H$,
$r(kG)$ and $r_w(H)$ are subrings of $r(H)$. The structure of $r_w(H)$ has been described in \cite{H.C}. By Lemma \ref{5.2}, $r(H)=r(kG)r_w(H)=r_w(H)r(kG)$. The injective map $\hat{G}\ra r(H)$, $\l\mapsto [V_1(\l)]=[V_{\l}]$ induces a ring embedding ${\mathbb Z}\hat{G}\hookrightarrow r(H)$ \cite{H.C}. In this case,
$r(kG)\cap r_w(H)={\mathbb Z}\hat{G}$.
Similarly, we have ${\mathbb Z}\hat{G}\subseteq G_0(kG)\subseteq G_0(H)$.
Moreover, $G_0(kG)=r(kG)$ since $kG$ is semisimple.
\end{remark}

\section{\bf An example}\selabel{5.3}

In this section, we apply the results of the previous sections to investigate the representations of the Hopf-Ore extensions of the group algebras of dihedral groups.

For any positive integer $n\>2$, the dihedral group $D_n$ of order $2n$ is defined by
$$D_n=\langle a,b|a^n=b^2=(ba)^2=1\rangle.$$

Throughout this section, assume that $n=2m$ is even and $m$ is odd with $m>1$. We also assume char$(k)=0$.
Let $\om\in k$ be a root of unity with the order $|\om|=n$.

In this case, $kD_n$ is semisimple and $a^m\in Z(D_n)$.

Let $\l, \chi\in\hat{D_n}$ be given by $\l(a)=1$, $\l(b)=-1$, $\chi(a)=-1$ and $\chi(b)=1$. Then $\hat{D_n}=\{\e, \l, \chi, \l\chi\}$ and $\hat{D_n}$ is isomorphic to the Klein group $K_4$. Therefore, $kD_n$ has 4 non-isomorphic one-dimensional simple modules $\{V_{\e}, V_{\l}, V_{\chi}, V_{\l\chi}\}$, where $V_{\e}$ is the trivial $kD_n$-module. There are $m-1$ non-isomorphic two-dimensional simple $kD_n$-modules $V_l$, $1\<l\<m-1$, their corresponding matrix representations $\rho_l: kD_n\ra M_2(k)$ are given by
$$\rho_l(a)=\left(\begin{array}{cc}
\om^l&0\\
0&\om^{-l}\\
\end{array}
\right) \text{ and } \rho_l(b)=\left(\begin{array}{cc}
0&1\\
1&0\\
\end{array}\right).$$
Let $I=\{\e, \l, \chi, \l\chi, 1, 2, \cdots, m-1\}$. Then $\{V_i|i\in I\}$ is a representative set of isomorphic classes of simple $kD_n$-modules.
The following lemma is well-known.

\begin{lemma}\label{6.1}
Let $1\<l, t\<m-1$ with $l\neq t$. Then the following hold:
\begin{enumerate}
\item[(1)] $V_{\l}\ot V_{\l}\cong V_{\chi}\ot V_{\chi}\cong V_{\e}$ and $V_{\l}\ot V_{\chi}\cong V_{\l\chi}$;
\item[(2)] $V_{\l}\ot V_l\cong V_l$ and $V_{\chi}\ot V_l\cong V_{m-l}$;
\item[(3)] if $l+t<m$ then $V_l\ot V_t\cong V_{|l-t|}\oplus V_{l+t}$;
\item[(4)] if $l+t=m$ then $V_l\ot V_t\cong V_{|l-t|}\oplus V_{\chi}\oplus V_{\l\chi}$;
\item[(5)] if $l+t>m$ then $V_l\ot V_t\cong V_{|l-t|}\oplus V_{n-l-t}$;
\item[(6)] if $2l<m$ then $V_l\ot V_l\cong V_{\e}\oplus V_{\l}\oplus V_{2l}$;
\item[(7)] if $2l>m$ then $V_l\ot V_l\cong V_{\e}\oplus V_{\l}\oplus V_{n-2l}$.
\end{enumerate}

\end{lemma}

Since $m$ is odd, $|\chi(a^m)|=2=|\chi|$. One can form a Hopf-Ore extension $kD_n(\chi, a^m, 0)$. Note that $\chi^{-1}=\chi$.

Throughout the rest of this section, let $H=kD_n(\chi, a^m, 0)$.

Let $I_0=\{\e, \l, 1, 2, \cdots, \frac{m-1}{2}\}$. Then it follows from Lemma \ref{6.1}(1, 2) and Corollary \ref{3.6}(2) that the following set is a representative set of isomorphic classes of finite dimensional simple $H$-modules:
$$\{V_i, V(j,\b)|i\in I, j\in I_0, \b\in k^{\ti}\}.$$
By Lemma \ref{6.1}(1, 2) and Corollary \ref{4.13}, the following set is a representative set of isomorphic classes of finite dimsensional indecomposable $H$-modules:
$$\{V_t(i), V_t(j,\b)|t\in\mathbb{N}^+, i\in I, j\in I_0, \b\in k^{\ti}\}.$$
Moreover, $V_t(\chi,\b)\cong V_t(\e,\b)$, $V_t(\l\chi,\b)\cong V_t(\l,\b)$
and $V_t(j,\b)\cong V_t(m-j,\b)$ for any $t\>1$, $1\<j\<m-1$ and $\b\in k^{\ti}$.

In what follows, we will frequently use the above two classifications, but not mention them for simplicity.

For any $i\in I$, it follows from Convention \ref{5.1} that there is a scale $\om_i\in k^{\ti}$ such that $a^mv=\om_iv$, $v\in V_i$. It is easy to see that either $\om_i=1$ or $\om_i=-1$. Since $s=|\chi|=2$, $\om_i^s=\om_i^2=1$. Thus, by Propositions \ref{5.3}-\ref{5.5}, we have the following corollary.

\begin{corollary}\label{6.2}
For any $M, N\in{\rm mod}H$, $M\ot N\cong N\ot M$. Consequently, $G_0(H)$ and $r(H)$ are both commutative rings.
\end{corollary}

\begin{convention}\label{6.3}
For any $V\in{\rm mod}H$ and $l\in\mathbb N$, define $V^{\ot l}$ by $V^{\ot 0}=V_1(\e)\cong V_{\e}$ for $l=0$, $V^{\ot 1}=V$ for $l=1$, and $V^{\ot l}=V\ot V\ot\cdots\ot V$, the tensor product of $l$-folds of $V$, for $l>1$.
\end{convention}

\subsection{\bf The Grothendieck ring of $H$}
In this subsection, we will investigate the Grotendieck ring $G_0(H)$. From Remark \ref{5.6}, ${\mathbb Z}\hat{D_n}\subseteq G_0(kD_n)\subset G_0(H)$. Moreover, $\e=1$, the identity of $G_0(H)$, and ${\mathbb Z}\hat{D_n}\cong{\mathbb Z}K_4$.

\begin{lemma}\label{6.4}
Let $1\<l\<m-1$. Then the decomposition of $V_1^{\ot l}$ is given as follows:
\begin{enumerate}
\item[(1)] if $l=2r-1$ is odd then $V_1^{\ot(2r-1)}\cong\oplus_{j=1}^r\binom{2r-1}{r-j}V_{2j-1}$;
\item[(2)] if $l=2r$ is even then $V_1^{\ot2r}\cong\binom{2r-1}{r-1}(V_{\e}\oplus V_{\l})\oplus(\oplus_{j=1}^r\binom{2r}{r-j}V_{2j})$.
\end{enumerate}
\end{lemma}

\begin{proof}
We prove the lemma by induction on $l$. For $l=1$, it is trivial. For $l=2$, it follows from Lemma \ref{6.1}(6). Now let $2<l\<m-1$. If $l=2r-1$ is odd, then by the induction hypothesis and Lemma \ref{6.1}, we have
$$\begin{array}{rcl}
V_1^{\ot(2r-1)}&\cong&V_1\ot V_1^{\ot(2r-2)}\\
&\cong&V_1\ot(\binom{2r-3}{r-2}(V_{\e}\oplus V_{\l})\oplus(\oplus_{j=1}^{r-1}\binom{2r-2}{r-1-j}V_{2j}))\\
&\cong&\binom{2r-3}{r-2}(V_1\ot V_{\e}\oplus V_1\ot V_{\l})\oplus(\oplus_{j=1}^{r-1}\binom{2r-2}{r-1-j}V_1\ot V_{2j})\\
&\cong&2\binom{2r-3}{r-2}V_1
\oplus(\oplus_{j=1}^{r-1}\binom{2r-2}{r-1-j}(V_{2j-1}\oplus V_{2j+1}))\\
%\cong&(2\binom{2r-3}{r-2}+\binom{2r-2}{r-2})V_1
%\oplus(\oplus_{j=2}^{r-1}(\binom{2r-2}{r-1-j}+\binom{2r-2}{r-j})V_{2j-1})
%\oplus\binom{2r-2}{0} V_{2r-1}))\\
&\cong&\oplus_{j=1}^r\binom{2r-1}{r-j}V_{2j-1}.
\end{array}$$
Similarly, if $l=2r$ is even then
$V_1^{\ot 2r}\cong\binom{2r-1}{r-1}(V_{\e}\oplus V_{\l})\oplus(\oplus_{j=1}^r\binom{2r}{r-j}V_{2j})$.
This completes the proof.
\end{proof}

Let $x=[V_1]$ in $G_0(kD_n)$. Then we have the following lemma.

\begin{lemma}\label{6.5}
Let $1\<l\<m-1$. Then the following hold in $G_0(kD_n)$:
\begin{enumerate}
\item[(1)] $\l x=x$;
\item[(2)] if $l=2r-1$ is odd, then
$$\begin{array}{c}
[V_{2r-1}]=\sum_{i=0}^{r-1}(-1)^{i}\frac{2r-1}{2r-1-2i}\binom{2r-2-i}{i}x^{2r-1-2i};
\end{array}$$
\item[(3)] if $l=2r$ is even, then
$$\begin{array}{c}
[V_{2r}]=\sum_{i=0}^{r-1}(-1)^{i}\frac{2r}{2r-i}\binom{2r-i}{i}x^{2r-2i}
+(-1)^r(\l+1).
    \end{array}$$
\end{enumerate}
\end{lemma}
\begin{proof}
Note that $\frac{2r-1}{2r-1-i}\binom{2r-2-i}{i}$ and $\frac{2r}{2r-i}\binom{2r-i}{i}$ are integers for all $0\<i\<r-1$.
Part (1) follows from Lemma \ref{6.1}(2). For Parts (2) and (3), we prove them by induction on $l$. If $l=1$ then it is trivial. If $l=2$ then it follows from Lemma \ref{6.1}(6). Now let $2<l\<m-1$. If $l=2r-1$ is odd, then by Lemma \ref{6.1}(3), the induction hypothesis and Part (1), we have
$$\begin{array}{rl}
[V_{2r-1}]=&x[V_{2r-2}]-[V_{2r-3}]\\
=&x(\sum_{i=0}^{r-2}(-1)^{i}\frac{2r-2}{2r-2-i}\binom{2r-2-i}{i}x^{2r-2-2i}
+(-1)^{r-1}(\l+1))\\
&-\sum_{i=0}^{r-2}(-1)^{i}\frac{2r-3}{2r-3-2i}\binom{2r-4-i}{i}x^{2r-3-2i}\\
=&\sum_{i=0}^{r-2}(-1)^{i}\frac{2r-2}{2r-2-i}\binom{2r-2-i}{i}x^{2r-1-2i}
+(-1)^{r-1}2x\\
&+\sum_{i=1}^{r-1}(-1)^i\frac{2r-3}{2r-1-2i}\binom{2r-3-i}{i-1}x^{2r-1-2i}\\
=&x^{2r-1}+\sum_{i=1}^{r-1}(-1)^{i}
(\frac{2r-2}{2r-2-i}\binom{2r-2-i}{i}+\frac{2r-3}{2r-1-2i}\binom{2r-3-i}{i-1})x^{2r-1-2i}\\
=&\sum_{i=0}^{r-1}(-1)^{i}\frac{2r-1}{2r-1-2i}\binom{2r-2-i}{i}x^{2r-1-2i}.\\
\end{array}$$
If $l=2r$ is even, then a similar argument shows that
$$\begin{array}{c}
[V_{2r}]=\sum_{i=0}^{r-1}(-1)^{i}\frac{2r}{2r-i}\binom{2r-i}{i}x^{2r-2i}
+(-1)^r(\l+1).\\
\end{array}$$
This completes the proof.
\end{proof}

\begin{corollary}\label{6.6}
The following hold:
\begin{enumerate}
\item[(1)] $G_0(kD_n)$ has a $\mathbb Z$-basis $X_1:=\{1, \l, \chi, \l\chi, x, x^2, \cdots, x^{m-1}\}$;
\item[(2)] $G_0(kD_n)$ is generated, as a ring, by its subring ${\mathbb Z}\hat{D_n}$ and the element $x$.
\end{enumerate}
\end{corollary}
\begin{proof}
(1) Since $\{[V_i]|i\in I\}$ is a $\mathbb Z$-basis of $G_0(kD_n)$, it follows from Lemma \ref{6.5}(2, 3) that $G_0(kD_n)$ is generated, as a $\mathbb Z$-module, by $X_1$. Since $\sharp\{[V_i]|i\in I\}=\sharp X_1$, $X_1$ is also a $\mathbb Z$-basis of $G_0(kD_n)$.

(2) It follows from (1).
\end{proof}

\begin{corollary}\label{6.7}
The following hold in $G_0(kD_n)$:
\begin{enumerate}
\item[(1)]
$x^m=\sum_{i=1}^{\frac{m-1}{2}}(-1)^{i-1}\frac{m}{m-2i}\binom{m-1-i}{i}x^{m-2i}
+(1+\l)\chi$;
\item[(2)]
$\chi x=\sum_{i=0}^{\frac{m-3}{2}}(-1)^{i}\frac{m-1}{m-1-i}\binom{m-1-i}{i}x^{m-1-2i}
+(-1)^{\frac{m-1}{2}}(1+\l)$.
\end{enumerate}
\end{corollary}
\begin{proof}
(1) By Lemma \ref{6.1}(4), $x[V_{m-1}]=[V_{m-2}]+\chi+\l\chi$. Since $m$ is odd, one gets from Lemma \ref{6.5} that
$$\begin{array}{rl}
x[V_{m-1}]=&x(\sum_{i=0}^{\frac{m-3}{2}}(-1)^{i}\frac{m-1}{m-1-i}\binom{m-1-i}{i}
x^{m-1-2i}+(-1)^{\frac{m-1}{2}}(\l+1))\\
=&\sum_{i=0}^{\frac{m-3}{2}}(-1)^{i}\frac{m-1}{m-1-i}\binom{m-1-i}{i}
x^{m-2i}+(-1)^{\frac{m-1}{2}}2x\\
=&\sum_{i=0}^{\frac{m-1}{2}}(-1)^{i}\frac{m-1}{m-1-i}\binom{m-1-i}{i}
x^{m-2i}\\
\end{array}$$
and
$$\begin{array}{rl}
[V_{m-2}]=&\sum_{i=0}^{\frac{m-3}{2}}(-1)^{i}\frac{m-2}{m-2-2i}\binom{m-3-i}{i}x^{m-2-2i}\\
=&\sum_{i=1}^{\frac{m-1}{2}}(-1)^{i-1}\frac{m-2}{m-2i}\binom{m-2-i}{i-1}x^{m-2i}.\\
\end{array}$$
It follows that
$$\begin{array}{rl}
&\sum_{i=0}^{\frac{m-1}{2}}(-1)^{i}\frac{m-1}{m-1-i}\binom{m-1-i}{i}
x^{m-2i}\\
=&\sum_{i=1}^{\frac{m-1}{2}}(-1)^{i-1}\frac{m-2}{m-2i}\binom{m-2-i}{i-1}x^{m-2i}+\chi+\l\chi.
\end{array}$$
Hence we have
$$\begin{array}{rl}
x^m=&\sum_{i=1}^{\frac{m-1}{2}}(-1)^{i-1}
(\frac{m-1}{m-1-i}\binom{m-1-i}{i}+\frac{m-2}{m-2i}\binom{m-2-i}{i-1})
x^{m-2i}+\chi+\l\chi\\
=&\sum_{i=1}^{\frac{m-1}{2}}(-1)^{i-1}
\frac{m}{m-2i}\binom{m-1-i}{i}x^{m-2i}+\chi+\l\chi.\\
\end{array}$$
(2) It follows from Lemmas \ref{6.1}(2) and \ref{6.5}(3).
\end{proof}

Let ${\mathbb Z}\hat{D_n}[x]$ be the polynomial ring in one variable $x$ over ${\mathbb Z}\hat{D_n}$. Define $f(x), g(x)\in {\mathbb Z}\hat{D_n}[x]$ by
$$\begin{array}{l}
f(x)=\sum_{i=0}^{\frac{m-3}{2}}(-1)^{i}\frac{m-1}{m-1-i}\binom{m-1-i}{i}x^{m-1-2i}
+(-1)^{\frac{m-1}{2}}(1+\l),\\
g(x)=\sum_{i=1}^{\frac{m-1}{2}}(-1)^{i-1}\frac{m}{m-2i}\binom{m-1-i}{i}x^{m-2i}
+(1+\l)\chi.\\
\end{array}$$
Let $J$ be the ideal of ${\mathbb Z}\hat{D_n}[x]$ generated by $\l x-x$,
$\chi x-f(x)$ and $x^m-g(x)$. Then we have the following proposition.

\begin{proposition}\label{6.8}
$G_0(kD_n)\cong {\mathbb Z}\hat{D_n}[x]/J$, the factor ring of ${\mathbb Z}\hat{D_n}[x]$ modulo $J$.
\end{proposition}
\begin{proof}
By Corollaries \ref{6.2} and \ref{6.6}(2), the ring embedding ${\mathbb Z}\hat{D_n}\hookrightarrow G_0(kD_n)$ can be extended to a ring epimorphism
$\phi: {\mathbb Z}\hat{D_n}[x]\ra G_0(kD_n)$ by $\phi(x)=[V_1]$. By Lemma \ref{6.5}(1) and Corollary \ref{6.7}, $\phi(J)=0$. Hence $\phi$ induces a ring epimorphism
$\ol{\phi}: {\mathbb Z}\hat{D_n}[x]/J\ra G_0(kD_n)$ given by $\ol{\phi}(\ol z)=\phi(z)$, where $\ol{z}$ denotes the image of $z\in {\mathbb Z}\hat{D_n}[x]$ under the canonical epimorphism ${\mathbb Z}\hat{D_n}[x]\ra{\mathbb Z}\hat{D_n}[x]/J$.
By the definition of $J$, ${\mathbb Z}\hat{D_n}[x]/J$ is generated, as a $\mathbb{Z}$-module, by $U:=\{\ol{1}, \ol{\l}, \ol{\chi}, \ol{\l}\ol{\chi}, \ol{x}, \ol{x}^2, \cdots, \ol{x}^{m-1}\}$. By Corollary \ref{6.6}(1), $\ol{\phi}(U)$ is a $\mathbb{Z}$-basis of $G_0(kD_n)$. It follows that $U$ is a $\mathbb{Z}$-basis of ${\mathbb Z}\hat{D_n}[x]/J$ and $\ol{\phi}$ is a ring isomorphism.
\end{proof}

\begin{remark}\label{6.9}
From Proposition \ref{6.8}, $G_0(kD_n)$ is a commutative ring generated by its subring ${\mathbb Z}\hat{D_n}$ and an element $x (:=[V_1])$ subject to the three relations given in Lemma \ref{6.5}(1) and Corollary \ref{6.7}.
\end{remark}

Let $y_{\b}=[V(\e,\b)]$ in $G_0(H)$ for any $\b\in k^{\ti}$. Then by Lemma \ref{5.2}, one gets the following lemma.

\begin{lemma}\label{6.10}
$G_0(H)$ is generated, as a ring, by $G_0(kD_n)\cup\{y_{\b}|\b\in k^{\ti}\}$.
\end{lemma}

\begin{lemma}\label{6.11}
Let $\a, \b\in k^{\ti}$ with $\a\neq-\b$. Then the following hold in $G_0(H)$:
$$\chi y_{\b}=y_{\b};\ y_{\a}y_{\b}=2y_{\a+\b};\ y_{\b}y_{-\b}=2(1+\chi).$$
\end{lemma}
\begin{proof}
By Lemma \ref{5.2} and Proposition \ref{3.4}, one knows that
$V_{\chi}\ot V(\e, \b)\cong V(\chi, \b)\cong V(\e, \b)$. Hence $\chi y_{\b}=y_{\b}$.
By Propositions \ref{5.4} and \ref{3.4}, we have
$V(\e, \a)\ot V(\e, \b)\cong V(\e, \a+\b)\oplus V(\chi, \a+\b)\cong 2V(\e, \a+\b)$.
Hence $y_{\a}y_{\b}=2y_{\a+\b}$. By Proposition \ref{5.4}, Remark \ref{4.1} and $|\chi|=2$, one gets $V(\e, \b)\ot V(\e, -\b)\cong V(\e, 0)\oplus V(\chi, 0)\cong V_2(\e)\oplus V_2(\chi)$. Then it follows from the proof of Proposition \ref{4.3} that
$y_{\b}y_{-\b}=[V_2(\e)]+[V_2(\chi)]=2(1+\chi)$.
\end{proof}

\begin{lemma}\label{6.12}
The set $X_2:=\{1,\l,\chi,\l\chi, x^l, \chi x^l|1\<l\<\frac{m-1}{2}\}$ is also a $\mathbb Z$-basis of $G_0(kD_n)$.
\end{lemma}
\begin{proof}
Let $N$ be the $\mathbb Z$-submodule of $G_0(kD_n)$ generated by $X_2$. Then by Lemma \ref{6.5}(1), $\l N=N$. Clearly, $\chi N=N$. By Lemma \ref{6.1}(2), $\chi[V_{\frac{m-1}{2}}]=[V_{\frac{m+1}{2}}]$. Then it follows from Lemma \ref{6.5}(2, 3) that $x^{\frac{m+1}{2}}\in N$. This implies $xN\subseteq N$ by $\l N=N$ and $\chi N=N$. Therefore, it follows from Corollary \ref{6.6}(2) that $N$ is an ideal of $G_0(kD_n)$, and so $N=G_0(kD_n)$ by $1\in N$. Thus, the lemma follows from $\sharp X_2=\sharp\{[V_i]|i\in I\}$.
\end{proof}

\begin{lemma}\label{6.13}
$G_0(H)$ has a $\mathbb{Z}$-basis $X_1\cup X_3$, where $X_1$ is the $\mathbb Z$-basis of $G_0(kD_n)$ given in Corollary \ref{6.6}(1) and $X_3:=\{\l y_{\b}, x^ly_{\b}|0\<l\<\frac{m-1}{2}, \b \in k^{\ti}\}$.
\end{lemma}
\begin{proof}
Let $L$ be the $\mathbb{Z}$-submodules of $G_0(H)$ generated by $\{[V(i,\b)]|i\in I,\b \in k^{\ti}\}$. Then $G_0(H)=G_0(kD_n)\oplus L$ as $\mathbb{Z}$-modules and $L$ has a $\mathbb Z$-basis $\{[V(i,\b)]|i\in I_0,\b \in k^{\ti}\}$. By Lemma \ref{5.2}, $L$ is a $G_0(kD_n)$-submodule of $G_0(H)$, and is generated, as a $G_0(kD_n)$-module, by $\{y_{\b}|\b\in k^{\ti}\}$. Hence it follows from Lemmas \ref{6.11} and \ref{6.12}, $L$ is generated, as a $\mathbb Z$-module, by $X_3$. It is left to show that $X_3$ is linearly independent over $\mathbb Z$. Note that $y_{\b}=[V(\e,\b)]$.
By Lemmas \ref{5.2}(2) and \ref{6.4}, $\l y_{\b}=[V(\l,\b)]$, $xy_{\b}=[V(1,\b)]$
and 
$$\begin{array}{c}
x^ly_{\b}\equiv [V(l,\b)] \text{ modulo } \mathbb{Z}[V(\e,\b)]+\mathbb{Z}[V(\l,\b)]+\sum_{i=1}^{l-1}\mathbb{Z}[V(i,\b)]
\end{array}$$
for all $2\<l\<\frac{m-1}{2}$. Since $\{[V(\e,\b)], [V(\l,\b)], [V(l,\b)]|1\<l\<\frac{m-1}{2}, \b \in k^{\ti}\}$ is linearly independent over $\mathbb{Z}$, so is $\{\l y_{\b}, x^l\omega_{\b}|0\<l\<\frac{m-1}{2},\b \in k^{\ti}\}$. This completes the proof.
\end{proof}

Let  $Y=\{y_{\b}|\b \in k^{\ti}\}$ and $G_0(kD_n)[Y]$ the polynomial ring in variables $Y$ over $G_0(kD_n)$. Put
$$U:=\{\chi y_{\b}-y_{\b}, y_{\a}y_{\b}-2y_{\a+\b}, y_{\b}y_{-\b}-2(1+\chi)|\a, \b\in k^{\ti} \text{ with } \a\neq -\b\},$$
and let $(U)$ be the ideal of $G_0(kD_n)[Y]$ generated by $U$.

\begin{theorem}\label{6.14}
$G_0(H)$ is isomorphic to the factor ring $G_0(kD_n)[Y]/(U)$.
\end{theorem}
\begin{proof}
By Corollary \ref{6.2} and Lemma \ref{6.10}, the ring embedding $G_0(kD_n)\hookrightarrow G_0(H)$ can be extended to a ring epimorphism $\phi: G_0(kD_n)[Y]\ra G_0(H)$ by $\phi(y_{\b})=[V(\e, \b)]$ for all $\b\in k^{\ti}$. By Lemma \ref{6.11}, $\phi(U)=0$. Hence $\phi$ induces a ring epimorphism $\ol{\phi}: G_0(kD_n)[Y]/(U)\ra G_0(H)$ given by $\ol{\phi}(\pi(z))=\phi(z)$ for any $z\in  G_0(kD_n)[Y]$, where $\pi:  G_0(kD_n)[Y]\ra  G_0(kD_n)[Y]/(U)$ is the canonical ring epimorphism. Clearly, $\pi(\chi y_{\b})=\pi(y_{\b})$ and
$G_0(kD_n)[Y]/(U)=\pi(G_0(kD_n))+\sum_{\b\in k^{\ti}}\pi(G_0(kD_n)y_{\b})$.
Then by Lemma \ref{6.12}, $\sum_{\b\in k^{\ti}}\pi(G_0(kD_n)y_{\b})$ is generated, as a $\mathbb Z$-module, by $Y_1:=\{\pi(\l y_{\b}), \pi(x^ly_{\b})|0\<l\<\frac{m-1}{2}, \b\in k^{\ti}\}$. Hence $G_0(kD_n)[Y]/(U)$ is generated, as a $\mathbb Z$-module, by $\pi(X_1)\cup Y_1$, where $X_1$ is the $\mathbb Z$-basis of $G_0(kD_n)$ given in Corollary \ref{6.6}(1). It is easy to check that $\ol{\phi}(z_1)\neq\ol{\phi}(z_2)$ for any $z_1\neq z_2$ in $\pi(X_1)\cup Y_1$ and that $\ol{\phi}(\pi(X_1)\cup Y_1)$ is a $\mathbb Z$-basis of $G_0(H)$ by Lemma \ref{6.13}. Hence $\pi(X_1)\cup Y_1$ is $\mathbb Z$-basis of $G_0(kD_n)[Y]/(U)$ and $\ol{\phi}$ is a ring isomorphism.
\end{proof}

\subsection{\bf The Green ring of $H$}

In this subsection, we will investigate the Green ring $r(H)$. By Remark \ref{5.6},
${\mathbb Z}\hat{D_n}\subset G_0(kD_n)=r(kD_n)\subset r(H)$. Moreover, $\e=1$, the identity of $r(H)$, and ${\mathbb Z}\hat{D_n}\cong{\mathbb Z}K_4$.

Let $R$ be the $\mathbb{Z}$-submodule of $r(H)$ generated by
$\{[V_t(i)]|i\in I,t\>1\}$. By Proposition \ref{5.5}(1), $R$ is a subring of $r(H)$.
Clearly, $G_0(kD_n)\subset R$. By Proposition \ref{5.2}(1), we have the following lemma.

\begin{lemma}\label{6.15}
$R$ is a free $G_0(kD_n)$-module with a basis $\{[V_t(\e)]|t\>1\}$.
\end{lemma}

By Proposition \ref{5.5}(1) or \cite[Theorem 3.15]{H.H}, we have the following lemma.

\begin{lemma}\label{6.16}
Let $t\>2$. Then the following hold:
\begin{enumerate}
\item[(1)] $V_2(\e)\ot V_1(\e)\cong V_2(\e)$ and $V_3(\e)\ot V_1(\e)\cong V_3(\e)$;
\item[(2)] if $t$ is even, then $V_2(\e)\ot V_t(\e)\cong V_t(\e)\oplus V_t(\chi)$;
\item[(3)] if $t$ is odd, then $V_2(\e)\ot V_t(\e)\cong V_{t+1}(\e)\oplus V_{t-1}(\chi)$;
\item[(4)] if $t\>3$, then $V_3(\e)\ot V_t(\e)\cong V_{t+2}(\e)\oplus V_{t-2}(\e)\oplus V_t(\chi)$.
\end{enumerate}
\end{lemma}

Let $y=[V_2(\e)]$ and $z=[V_3(\e)]$ in $r(H)$. Then $y, z\in R$. For any $t\>1$, let $M_t$ be the $G_0(kD_n)$-submodule of $R$ generated by $\{[V_l(\e)]|1\<l\<t\}$. Let $M_{-1}=M_0=0\subset R$. Then $M_{t-1}\subset M_t$ for all $t\>0$.

\begin{corollary}\label{6.17}
Let $t\>1$. Then the following hold:
\begin{enumerate}
\item[(1)] $M_t$ has a $\mathbb Z$-basis $\{[V_l(i)]|i\in I, 1\<l\<t\}$;
\item[(2)] $yM_t\subseteq M_{t+1}$ if $t$ is odd and $yM_t\subseteq M_t$ if $t$ is even;
\item[(3)] $zM_t\subseteq M_{t+2}$.
\end{enumerate}
\end{corollary}
\begin{proof}
(1) follows from Lemma \ref{6.15}, Proposition \ref{5.2}(1) and the fact that $\{[V_i]|i\in I\}$ is a $\mathbb Z$-basis of $G_0(kD_n)$. (2) and (3) follows from Lemma \ref{6.16} and (1).
\end{proof}

\begin{lemma}\label{6.18}
The following hold:
\begin{enumerate}
\item[(1)] $y^2=(1+\chi)y$ in $R$ (or $r(H)$);
\item[(2)] $R$ is generated, as a ring, by $G_0(kD_n)\cup\{y,z\}$.
\end{enumerate}
\end{lemma}
\begin{proof}
(1) It follows from Lemma \ref{6.16}(2) and Proposition \ref{5.2}(1).

(2) Let $R'$ be the subring of $R$ generated by $G_0(kD_n)\cup\{y,z\}$. By Lemma \ref{6.15}, we only need to show $[V_t(\e)]\in R'$ for all $t\>1$. Clearly, $[V_t(\e)]\in R'$ for $1\<t\<3$. By Lemma \ref{6.16}(3) and Proposition \ref{5.2}(1), $[V_4(\e)]=yz-\chi y\in R'$. Now let $t>4$ and assume $[V_l(\e)]\in R'$ for all $1\<l\<t-1$. By Lemma \ref{6.16}(4), $V_3(\e)\ot V_{t-2}(\e)\cong V_t(\e)\oplus V_{t-4}(\e)\oplus V_{t-2}(\chi)$. Hence $[V_t(\e)]=(z-\chi)[V_{t-2}(\e)]-[V_{t-4}(\e)]\in R'$ by Proposition \ref{5.2}(1) and the induction hypothesis.
This completes the proof.
\end{proof}

\begin{lemma}\label{6.19}
Let $t\>0$. Then the following hold:
\begin{enumerate}
\item[(1)] $z^t\equiv[V_{2t+1}(\e)]$ modulo $M_{2t-1}$;
\item[(2)] $yz^t\equiv[V_{2t+2}(\e)]$ modulo $M_{2t}$;
\item[(3)] $\{z^t, yz^t|t\>0\}$ is a $G_0(kD_n)$-basis of $R$.
\end{enumerate}
\end{lemma}
\begin{proof}
(1) It is trivial for $t=0, 1$. Now let $t>1$ and assume $z^{t-1}\equiv[V_{2t-1}(\e)]$ modulo $M_{2t-3}$. Then $z^{t-1}=[V_{2t-1}(\e)]+u$ for some $u\in M_{2t-3}$.
Thus, by Lemma \ref{6.16}(4) and Corollary \ref{6.17}(3),
$$\begin{array}{rcl}
z^t&=&z[V_{2t-1}(\e)]+zu\\
&=&[V_{2t+1}(\e)]+[V_{2t-3}(\e)]+[V_{2t-1}(\chi)]+zu\\
&\equiv&[V_{2t+1}(\e)]\ \text{ modulo } M_{2t-1}.\\
\end{array}$$

(2) By Corollary \ref{6.17}(2), $yM_{2t-1}\subseteq M_{2t}$. Then by (1) and Lemma \ref{6.16}(1, 3),  $yz^t\equiv y[V_{2t+1}(\e)]\equiv[V_{2t+2}(\e)]$ module $M_{2t}$.

(3) It follows from (1), (2) and Lemma \ref{6.15}.
\end{proof}

\begin{proposition}\label{6.20}
Let $G_0(kD_n)[y,z]$ be the polynomial ring in two variables $y, z$ over $G_0(kD_n)$. Then $R\cong G_0(kD_n)[y,z]/(y^2-(1+\chi)y)$, where $(y^2-(1+\chi)y)$ is the ideal of $G_0(kD_n)[y,z]$ generated by $y^2-(1+\chi)y$.
\end{proposition}
\begin{proof}
By Corollary \ref{6.2} and Lemma \ref{6.18}(2), the ring embedding $G_0(kD_n)\hookrightarrow R$ can be extended to a ring epimorphism $\phi: G_0(kD_n)[y,z]\ra R$ such that $\phi(y)=[V_2(\e)]$ and $\phi(z)=[V_3(\e)]$. By Lemma \ref{6.18}(1), $\phi$ induces a ring epimorphism
$\ol{\phi}: G_0(kD_n)[y,z]/(y^2-(1+\chi)y)\ra R$ such that $\ol{\phi}(\pi(u))=\phi(u)$
for any $u\in G_0(kD_n)[y,z]$, where $\pi: G_0(kD_n)[y,z]\ra G_0(kD_n)[y,z]/(y^2-(1+\chi)y)$ is the canonical epimorphism.
In an obvious way, $G_0(kD_n)[y,z]/(y^2-(1+\chi)y)$ becomes a $G_0(kD_n)$-module. In this case, $\ol{\phi}$ is a $G_0(kD_n)$-module map.
Clearly, $G_0(kD_n)[y,z]/(y^2-(1+\chi)y)$ is generated, as a $G_0(kD_n)$-module, by
$X_4:=\{\pi(z^t), \pi(yz^t)|t\>0\}$. By Lemma \ref{6.19}(3), $\ol{\phi}(X_4)$ is a  $G_0(kD_n)$-basis of $R$. This implies that $\ol{\phi}$ is injective, and so it is a ring isomorphism.
\end{proof}

For any $\b \in k^{\ti}$, let $w_{\b}=[V(\e,\b)]$ in $r(H)$. Then we have the following lemma.

\begin{lemma}\label{6.21}
$r(H)$ is generated, as a ring, by $R\cup\{w_{\b}|\b\in k^{\ti}\}$.
\end{lemma}
\begin{proof}
Let $R'$ be the subring of $r(H)$ generated by $R\cup\{w_{\b}|\b\in k^{\ti}\}$.
Then $G_0(kD_n)\subset R\subset R'$. By Lemma \ref{5.2}(2) and the classification of finite dimensional indecomposable $H$-modules, it is enough to show that
$[V_t(\e,\b)]\in R'$ for all $t\>1$ and $\b\in k^{\ti}$.
We prove it by induction on $t$. For $t=1$, it is trivial. Now assume $t\>1$ and assume $[V_{l}(\e,\b)]\in R'$ for all $1\<l\<t$ and $\b\in k^{\ti}$. By Proposition \ref{5.3} (or \cite[Theorem 3.6]{H.H}),
$V_3(\e)\ot V_t(\e,\b)\cong V_t(\e,\b)\oplus V_{t-1}(\e,\b)\oplus V_{t+1}(\e,\b)$, where $V_0(\e,\b)=0$. Hence $[V_{t+1}(\e,\b)]=(z-1)[V_t(\e, \b)]-[V_{t-1}(\e, \b)]\in R'$.
\end{proof}

\begin{lemma}\label{6.22}
Let $\a, \b\in k^{\ti}$ with $\a\neq-\b$. Then the following hold in $r(H)$:
$$\chi w_{\b}=w_{\b};\ w_{\a}w_{\b}=2w_{\a+\b};\ w_{\b}w_{-\b}=(1+\chi)y;\ yw_{\b}=2w_{\b}.$$
\end{lemma}

\begin{proof}
The first three equations follow from an argument similar to the proofs of Lemma \ref{6.11} and Lemma \ref{5.2}(1). By Proposition \ref{5.3} (or \cite[Theorem 3.6]{H.H}),
$V_2(\e)\ot V(\e,\b)\cong 2V(\e,\b)$, and hence $yw_{\b}=2w_{\b}$.
\end{proof}

Let $P$ be the $\mathbb Z$-submodule of $r(H)$ generated by $\{[V_t(i,\b)]|t\>1, i\in I, \b\in k^{\ti}\}$. Then $r(H)=R\oplus P$ and $P$ has a $\mathbb Z$-basis $\{[V_t(i,\b)]|t\>1, i\in I_0, \b\in k^{\ti}\}$.

For $t\>1$ and $\b\in k^{\ti}$, let $P^{\b}$ and $P^{\b}_t$ be the $\mathbb Z$-submodules of $r(H)$ generated by $\{[V_l(i,\b)]|l\>1, i\in I\}$ and $\{[V_l(i,\b)]|1\<l\<t, i\in I\}$, respectively. Then $P^{\b}$ and $P^{\b}_t$ have the $\mathbb Z$-bases $\{[V_l(i,\b)]|l\>1, i\in I_0\}$ and $\{[V_l(i,\b)]|1\<l\<t, i\in I_0\}$, respectively. Clearly, $P^{\b}_t\subset P^{\b}_{t+1}$, $P^{\b}=\sum_{t\>1}P^{\b}_t=\cup_{t\>1}P^{\b}_t$ and $P=\oplus_{\b\in k^{\ti}}P^{\b}$.
By Proposition \ref{5.3}, $P$ and $P^{\b}$ are $R$-submodules of $r(H)$.

Let $t\>1$ and $\b\in k^{\ti}$. By Lemma \ref{5.2}(2), $P^{\b}_t$ is a $G_0(kD_n)$-module generated by $\{[V_l(\e,\b)]|1\<l\<t\}$. By Proposition \ref{5.3}, $V_3(\e)\ot V_t(i,\b)\cong V_t(i,\b)\oplus V_{t-1}(i,\b)\oplus V_{t+1}(i,\b)$ for any $i\in I$, where $V_0(i,\b)=0$. Hence $zP^{\b}_t\subseteq P^{\b}_{t+1}$. We claim that
$$z^lw_{\b}\equiv [V_{l+1}(\e,\b)] \text{ modulo } P^{\b}_l, \forall l\>1.$$
For $l=1$, $zw_{\b}=[V_2(\e, \b)]+[V_1(\e, \b)]\equiv[V_2(\e, \b)]$ modulo $P^{\b}_1$. Let $l>1$ and assume $z^{l-1}w_{\b}\equiv[V_l(\e, \b)]$ modulo $P^{\b}_{l-1}$. Then $z^{l-1}w_{\b}=[V_l(\e, \b)]+u$ for some $u\in P^{\b}_{l-1}$. Hence
$z^lw_{\b}=z[V_l(\e,\b)]+zu=[V_{l+1}(\e,\b)]+[V_l(\e,\b)]+[V_{l-1}(\e,\b)]+zu
\equiv [V_{l+1}(\e,\b)]$ modulo $P^{\b}_l$. Thus, we have shown the claim. Therefore,
$P^{\b}_t$ is generated, as a $G_0(kD_n)$-module, by $\{z^lw_{\b}|0\<l\<t-1\}$.
Then by Lemmas \ref{6.12} and \ref{6.22}, $P^{\b}_t$ is generated, as a $\mathbb Z$-module, by $X^{\b}_t:=\{\l z^lw_{\b}, x^iz^lw_{\b}|0\<i\<\frac{m-1}{2},0\<l\<t-1\}$.
Since $\sharp X^{\b}_t=\sharp\{[V_l(i,\b)]|1\<l\<t, i\in I_0\}$, $X^{\b}_t$ is also a $\mathbb Z$-basis of $P^{\b}_t$. It follows that $X^{\b}:=\{\l z^lw_{\b}, x^iz^lw_{\b}|0\<i\<\frac{m-1}{2}, l\>0\}$ is a $\mathbb Z$-basis of $P^{\b}$.
Summarizing the above discussion, we have the following lemma.

\begin{lemma}\label{6.23} $P$ has a $\mathbb{Z}$-basis
$B_P:=\{\l z^tw_{\b}, x^lz^tw_{\b}|0\<l\<\frac{m-1}{2}, t\>0, \b\in k^{\ti}\}$.
\end{lemma}

\begin{theorem}\label{6.24}
Let $R[Z]$ be the polynomial ring in the variables $Z=\{w_{\b}|\b \in k^{\ti}\}$ over $R$. Let $(W)$ be the ideal of $R[Z]$ generated by
$$W:=\left\{
\begin{array}{c}
\chi w_{\b}-w_{\b},\ w_{\a}w_{\b}-2w_{\a+\b},\\
w_{\b}w_{-\b}-(1+\chi)y,\ yw_{\b}-2w_{\b}\\
\end{array}\right|
\left.\begin{array}{c}
\a, \b\in k^{\ti}\\
\text{with } \a\neq\b\\
\end{array}\right\}.$$
Then $r(H)$ is isomorphic to the factor ring $R[Z]/(W)$.
\end{theorem}
\begin{proof}
By Lemma \ref{6.21}, the ring embedding $R\hookrightarrow r(H)$ can be extended to a ring epimorphism $\phi: R[Z]\ra r(H)$ such that $\phi(w_{\b})=[V(\e, \b)]$ for all $\b\in k^{\ti}$. By Lemma \ref{6.22}, $\phi$ induces a ring epimorphism
$\ol{\phi}: R[Z]/(W)\ra r(H)$ such that $\ol{\phi}(\pi(u))=\phi(u)$
for any $u\in R[Z]$, where $\pi: R[Z]\ra R[Z]/(W)$ is the canonical epimorphism.
By the definition of $W$, $R[Z]/(W)=\pi(R)+\sum_{\b\in k^{\ti}}\pi(Rw_{\b})$.
Let $X_2$ be the $\mathbb Z$-basis of $G_0(kD_n)$ given in Lemma \ref{6.12}.
Then by Lemma \ref{6.19}(3), $R$ has a $\mathbb Z$-basis $B_R:=\{rz^t, ryz^t|r\in X_2, t\>0\}$. Again by the definition of $W$, $\sum_{\b\in k^{\ti}}\pi(Rw_{\b})$ is generated, as a $\mathbb Z$-module, by
$$S_R:=\{\pi(\l z^tw_{\b}), \pi(x^lz^tw_{\b})|0\<l\<\frac{m-1}{2}, t\>0, \b\in k^{\ti}\}.$$
Hence $R[Z]/(W)$ is generated, as a $\mathbb Z$-module, by $B:=\pi(B_R)\cup S_R$.
From $\ol{\phi}(\pi(w_{\b}))=[V(\e, \b)]$ and $\ol{\phi}(\pi(r))=r$ for any $r\in R$, one can check that $\ol{\phi}(a)\neq\ol{\phi}(b)$ for any $a, b\in B$ with $a\neq b$,
and that $\ol{\phi}(B)=B_R\cup B_P$, which is a $\mathbb Z$-basis of $r(H)$ by Lemma \ref{6.23}. It follows that $\ol{\phi}$ is a ring isomorphism.
\end{proof}

Let $X:=\{x, y, z, w_{\b}|\b\in k^{\ti}\}$ and ${\mathbb Z}\hat{D_n}[X]$ the polynomial ring in variables $X$ over ${\mathbb Z}\hat{D_n}$.
Let $(Q)$ be the ideal of ${\mathbb Z}\hat{D_n}[X]$ generated by the following set
$$Q:=\left\{\left.
\begin{array}{c}
\chi x-f(x),\ x^m-g(x),\ y^2-(1+\chi)y,\\
\l x-x,\ \chi w_{\b}-w_{\b},\ yw_{\b}-2w_{\b},\\
 w_{\a}w_{\b}-2w_{\a+\b},\ w_{\b}w_{-\b}-(1+\chi)y\\
\end{array}\right|
\begin{array}{c}
\a, \b\in k^{\ti}\\
\text{with } \a\neq\b\\
\end{array}\right\},$$
where $f(x), g(x)\in {\mathbb Z}\hat{D_n}[x]\subset {\mathbb Z}\hat{D_n}[X]$ are given before Proposition \ref{6.8}. Then by Propositions \ref{6.8}, \ref{6.20} and Theorem \ref{6.24}, one gets the following corollary.

\begin{corollary}
$r(H)$ is isomorphic to the factor ring ${\mathbb Z}\hat{D_n}[X]/(Q)$.
\end{corollary}

\centerline{ACKNOWLEDGMENTS}

This work is supported by NNSF of China (No. 12071412).\\

\end{document}